\documentclass[12pt, twoside]{article}
\usepackage{amsmath,amsthm,amssymb}
\usepackage{times}
\usepackage{enumerate}
\usepackage{bm}
\usepackage[all]{xy}
\usepackage{mathrsfs}
\usepackage{amscd}
\usepackage{color}
\usepackage{comment}
\usepackage{slashed}

\def\titlerunning#1{\gdef\titrun{#1}}
\makeatletter
\def\author#1{\gdef\autrun{\def\and{\unskip, }#1}\gdef\@author{#1}}
\def\address#1{{\def\and{\\\hspace*{18pt}}\renewcommand{\thefootnote}{}%
		\footnote {#1}}%
	\markboth{\autrun}{\titrun}}
\makeatother
\def\email#1{e-mail: #1}
\def\subjclass#1{{\renewcommand{\thefootnote}{}%
		\footnote{\emph{Mathematics Subject Classification (2020):} #1}}}
\def\keywords#1{\par\medskip
	\noindent\textbf{Keywords.} #1}

\newtheorem{theorem}{Theorem}[section]
\newtheorem{corollary}[theorem]{Corollary}
\newtheorem{lemma}[theorem]{Lemma}
\newtheorem{proposition}[theorem]{Proposition}
\theoremstyle{definition}

\newtheorem{remark}[theorem]{Remark}

\numberwithin{equation}{section}

\xyoption{all}

\def \a {\alpha }
\def \b {\beta}

\def \de {\delta}

\def \la {\lambda}

\def\na {\nabla}

\begin{document}
	\baselineskip=17pt
	
	\titlerunning{Bochner–Yano Type Theorems for Conformal  Killing Vector Fields }
	\title{ Bochner–Yano Type Theorems for Conformal Killing Vector  Fields under Curvature Pinching Conditions} 
	
	\author{Teng Huang, Qiang Tan and Weiwei Wang}
	
	\date{}
	
	\maketitle
	
\address{Teng Huang: School of Mathematical Sciences, CAS Key Laboratory of Wu Wen-Tsun Mathematics, University of Science
and Technology of China, Hefei, Anhui, 230026, People’s Republic of China; \email{htmath@ustc.edu.cn; htustc@gmail.com}}
\address{Qiang Tan: School of Mathematical Sciences, Jiangsu University, Zhenjiang, Jiangsu, 212013, People’s Republic of China; \email{tanqiang@ujs.edu.cn}}
\address{Weiwei Wang: School of Mathematical Sciences, University of Electronic Science and Technology of China, Chengdu, Sichuan, 611731, People’s Republic of China; \email{wawnwg123@163.com; wangweiwei123@uestc.edu.cn}}
	\subjclass{53C20;53C21;53C24;58A10}

\begin{abstract}
In this article, we investigate conformal Killing vector fields on closed Riemannian manifolds under a curvature pinching condition. By establishing a new Bochner-type identity for the $1$‑form dual to a conformal Killing vector field, we derive a sharp gradient estimate via a Moser iteration procedure. Based on this estimate, we prove that, under a suitable upper bound on the Ricci curvature, every nontrivial conformal Killing vector field must be nowhere vanishing. Consequently, on even-dimensional manifolds with non-zero Euler characteristic, every conformal Killing vector field vanishes identically, which in turn implies that the conformal transformation group of such a manifold is finite. Our results extend the classical rigidity theorems of Yano and Bochner from the setting of non-positive Ricci curvature to that of small positive Ricci curvature, and generalize the recent results of Chen and Han from Killing vector fields to conformal Killing vector fields.
\end{abstract}
\keywords{Conformal Killing vector field, Bochner technique, Moser iteration, Gradient estimate, Ricci curvature.}

\section{Introduction}
Let $(M,g)$ be a connected, closed $n$-dimensional smooth Riemannian manifold. A smooth vector field $X$ is called a \textbf{conformal Killing vector field} if the Lie derivative of the metric $g$ with respect to $X$ satisfies
$$
\mathcal{L}_Xg=h g
$$
for some smooth function $h \in C^\infty(M)$. Taking the trace on  both sides, it follows that 
$$h=\frac{2}{n}\mathrm{div}\, X.$$ In the special case $h=0$, $X$ is a \textbf{Killing vector field}. The conformal Killing vector field and the induced conformal Killing operator, which were already studied in the context of Stein–Weiss operators \cite{Brans, SW}, play a fundamental role in the analysis of the Einstein–Lichnerowicz constraint system in general relativity (see \cite{Premos, Gourg}).

 A classical theorem of Yano \cite{Yano1} states that if the Ricci curvature is non-positive ($\mathrm{Ric} \leq 0$), then every conformal Killing vector field  is parallel.  Furthermore, when $\mathrm{Ric}<0$, every conformal Killing vector field on $M$ is identically zero.  Yano's result extends the earlier work of Bochner \cite{Boc}, who proved the corresponding rigidity for Killing vector fields under the same curvature assumptions. Thus, Yano's theorem can be viewed as a natural generalization of Bochner's theorem to the conformal setting.

The main purpose of this article is to generalize these classical rigidity results by allowing the Ricci curvature to be bounded above by a small, explicitly determined positive constant. More precisely, we prove the following theorem.
\begin{theorem}\label{T2}
	Let $(M,g)$ be a connected, closed smooth Riemannian manifold of dimension $n\geq2$. Let $\kappa, K,D,p$ be positive constants with $p>n/2$. Define $\lambda:=\lambda(n,p,\kappa,K,D)$ by
	$$\sqrt{\lambda D^2}= \min\left\{ 
		 \left(\frac{\tilde{C}(n,p)}{1 + \sqrt{K}D }e^{-(n-1)\sqrt{\kappa D^2}} \right)^{\frac{2pn}{2p-n}},  
\; e^{-(n-1)\sqrt{\kappa D^2}}
		\right\},$$
where $\tilde{C}(n,p)$ is a constant depending only on $n$ and $p$.

		Suppose that 
	\[
	-(n-1)\kappa\leq \operatorname{Ric}\leq \lambda , \quad
	\|R\|_{2p} \leq K,\quad 
	\operatorname{diam}(M) \leq D.
	\]
Then every conformal Killing vector field  on $M$ is either identically zero or nowhere vanishing.
\end{theorem}
Since Yano's work \cite{Yano1}, rigidity theorems of this type for conformal Killing vector fields have been extensively studied under different curvature conditions. Tanno and Weber \cite{TW} established that a connected compact Riemannian manifold with positive constant scalar curvature is globally isometric to a sphere if the Euler characteristic $\chi(M)\neq0$ and admits a closed conformal Killing vector field. Ejiri \cite{Ejiri} proved that for any $n>2$, there exists a closed Riemannian $n$-manifold of constant scalar curvature admitting a nowhere vanishing conformal Killing vector field. Gursky \cite{Gursky} showed that a $4$-manifold with negative Yamabe invariant admits no nontrivial conformal Killing vector field, provided an integral condition on the Weyl curvature is satisfied. We refer the reader to  \cite{Deshmukh1,SD,Yano2}  and the references therein for a more comprehensive overview of related results.

In odd dimensions, we always have  the Euler characteristic $\chi(M)=0$, so nowhere vanishing conformal Killing  vector fields are topologically possible. However, in even dimensions,  $\chi(M)\neq 0$ is permitted.  By the Poincar\'e--Hopf theorem, if  $\chi(M)\neq0$, then every vector field on $M$ must have at least one zero. Consequently, Theorem \ref{T2} immediately yields the following result:
\begin{theorem}\label{T3}
Let $(M,g)$ be a connected, closed $2n$-dimensional smooth Riemannian manifold with nonvanishing Euler characteristic. Under the conditions of Theorem \ref{T2},  every conformal Killing vector field on $M$ is identically zero.
\end{theorem}
For a closed Riemannian manifold $(M,g)$ with dimension $n \geq 2$, let $\operatorname{Con}(M,g)$ and $\operatorname{Iso}(M,g)$ denote the group of conformal transformations and the group of isometries, respectively. Both are finite dimensional Lie groups, and $\operatorname{Iso}(M,g)$ is a closed subgroup of  $\operatorname{Con}(M,g)$. The Lie algebra of $\operatorname{Con}(M,g)$ is spanned by conformal Killing vector fields, while the Lie algebra of $\operatorname{Iso}(M,g)$ is spanned by Killing vector fields. The relationship  between $\operatorname{Con}(M,g)$ and $\operatorname{Iso}(M,g)$ has been extensively studied in \cite{GK1,GK2,Obata}.  In the language of conformal transformation groups, Theorem \ref{T3} can be restated as follows.
\begin{corollary}\label{C1}
Let $(M,g)$ be a connected, closed $2n$-dimensional smooth Riemannian manifold with nonvanishing Euler characteristic. Under the conditions of Theorem \ref{T2}, the conformal transformation group $\operatorname{Con}(M,g)$ is finite.
\end{corollary}
Results of this type for the isometry group $\operatorname{Iso}(M,g)$ are obtained in \cite{CH}. Corollary \ref{C1} extends them to the conformal transformation group  $\operatorname{Con}(M,g)$.  
 
To analyze the group action of $\operatorname{Con}(M,g)$, we recall that 
a smooth action of a Lie group $G$ on a manifold $M$ is said to be locally free if the stabilizer subgroup 
$$G_p = \{x \in G \mid x \cdot p = p\}$$ 
is discrete for every point $p \in M$. Since $M$ is closed, $\text{Con}(M,g)$ is a finite dimensional Lie group acting smoothly on $M$. Examining its stabilizer subgroups  $\operatorname{Con}_p$, we obtain the following consequence.
\begin{corollary}\label{C3}
Let $(M,g)$ be a connected, closed smooth Riemannian manifold of dimension $n\geq2$. Under the conditions of Theorem \ref{T2}, the action of conformal transformation group $\operatorname{Con}(M,g)$ on $M$ is locally free.
\end{corollary}

In the special case $h=0$, the vector field $X$ reduces to a Killing vector field. We now revisit this case.

Bott proved in \cite{Bot} that the existence of a non vanishing Killing vector field implies that all Pontryagin numbers of $M$ vanish.  For further rigidity theorems for Killing vector fields under different curvature conditions, see \cite{Fuh, KN,LR} and the references therein. Combining Bott's result with Theorem \ref{T2} yields the following corollary.
	\begin{corollary}\label{C2}
	Let $(M,g)$ be a connected, closed $n$-dimensional smooth Riemannian manifold. Suppose that $M$ satisfies one of the following topological  conditions:
	\begin{itemize}\vspace{-0.4\baselineskip}
	\item[(1)] $n \equiv 0 \pmod 2$, and the Euler characteristic $\chi(M) \neq 0$;\vspace{-0.6\baselineskip}
	\item[(2)] $n \equiv 0 \pmod 4$, and at least one Pontryagin number is non-zero.
	\end{itemize}\vspace{-0.4\baselineskip}
		Then, under the conditions of Theorem \ref{T2}, the isometry group of the Riemannian metric $g$ on $M$ is finite.
	\end{corollary}
	\begin{remark}
	In \cite{CH}, Chen and Han proved that if $(M,g)$ is a closed Riemannian manifold satisfying one of the following topological assumptions:
\begin{itemize}\vspace{-0.4\baselineskip}
	\item[(1)] the Euler characteristic $\chi(M) \neq 0$;\vspace{-0.6\baselineskip}
	\item[(2)] the signature $\sigma(M) \neq 0$; \vspace{-0.6\baselineskip}
	\item[(3)] or, in the spin case, the elliptic genus is nonzero,
\end{itemize}\vspace{-0.4\baselineskip}
and if the Ricci curvature satisfies a ``small'' positive upper bound of the form
\[
-\lambda_1 \leq \operatorname{Ric}(g) \leq \varepsilon, \quad \operatorname{diam}(g) \leq 1,
\]
with an additional integral curvature bound $\|R\|_{p}\leq \lambda_2$ in the non-K\"ahler settings (see \cite[Theorems 1.1--1.3]{CH} for details), then the isometry group $\operatorname{Iso}(M,g)$ is finite.  Their proof relies on the Atiyah--Singer index theorem and the rigidity of certain Dirac operators, showing that no nontrivial Killing vector field can exist under these assumptions. Furthermore, in the case $n \equiv 0 \pmod 4$, Corollary \ref{C2} improves their Theorem 1.2: the latter requires the signature  $\sigma(M) \neq 0$, while our result only requires that at least one Pontryagin number be non-zero. 
\end{remark}
	
Our main results, Theorems \ref{T2} and \ref{T3}, extend those of Bochner \cite{Boc} and Yano \cite{Yano1} in two directions: generalizing Killing vector fields to conformal Killing vector fields (i.e. from the isometry group $\operatorname{Iso}(M,g)$ to the conformal transformation group $\operatorname{Con}(M,g)$); and relaxing the curvature assumption from negative Ricci curvature to Ricci curvature with small positive upper bounds. Specifically, under these curvature assumptions, the non-vanishing of the Euler characteristic forces every conformal Killing vector field to vanish identically, which implies the finiteness of $\operatorname{Con}(M,g)$. 

Similar extension results, such as \cite{CH,LR}, are restricted to Killing vector fields and do not consider the conformal setting.
 In Chen and Han's approach \cite{CH}, their proof relies on the index theorem or Dirac operators; instead, we use a purely analytic method here.  Moreover, the upper bound on the Ricci curvature in our results is explicitly given (see Theorem \ref{T2}), whereas their corresponding bound  is obtained via an existence argument. Our analytic method goes back to Le Couturier and Robert  \cite{LR},  but their Harnack inequality cannot be directly applied to conformal Killing vector fields. We extend their approach by using a new Bochner technique and a modified Moser iteration procedure to establish a Harnack inequality suitable for the conformal setting.

\noindent\textbf{Outline of the proof}.
In what follows, our main purpose is to prove Theorem \ref{T2}. Throughout, the dimension of the manifold is assumed to be at least 2. We outline our proof strategy in two parts:

First, following the strategy developed in \cite{LR,PS}, we apply  a  Bochner technique to compute $\Delta |\nabla\theta|^2$ for the dual $1$-form $\theta$, in contrast to the approach which computes $\Delta_{d} |\theta|^2$.  The resulting Bochner identity for $\theta$ is given in Lemma \ref{K2}, whose proof relies on the preliminaries established in Section 2.1.

Second, in Proposition \ref{P6}, we modify the Moser iteration procedure from Aubry \cite[Proposition 4.2]{Aub} (see also \cite{ACGR,HW}) to establish an $L^{\infty}$ bound for $\nabla\theta$. The main difference from Aubry's argument lies in the derivation of the integral inequality in Lemma \ref{L7}. 
This  $L^{\infty}$ bound for $\nabla\theta$ yields the Harnack inequality stated in Theorem \ref{T4}. Based on this, we  prove Theorem \ref{T2} and all other results in Section 4.

\noindent\textbf{Notations}\\
Throughout this article, $\langle \cdot,\cdot\rangle$ stands for the inner product induced by the  Riemannian metric $g$. Let $\mathfrak{X}(M)$ be the set of all smooth vector fields on  $M$. We denote by  $R$ the curvature tensors of $M$, and by ${\rm{Ric}}$ the Ricci curvature of $M$.  The symbol $ \mathrm{Ric}^{-} $ denotes the lowest eigenvalue of the Ricci curvature $\mathrm{Ric}$, i.e., for each $p\in M$,
\[\mathrm{Ric}^{-}(p):=\min_{v\in T_{p}M,|v|=1}\mathrm{Ric}(v,v),\]
and we set
\[\underline{\mathrm{Ric}}^{-}:=\max\{-\mathrm{Ric}^{-},0 \}. \]
The normalized $L^{p}$-norms are defined by
$$\|f\|_{p}=\bigg{(} \frac{1}{{\rm{Vol}}(g)}\int_{M}|f|^{p}\bigg{)}^{\frac{1}{p}}.$$
and for a $1$-form $\theta$, $\|\theta\|_p=\||\theta|\|_p$.
We also let $H^1(M)$ be the Sobolev space of functions on $M$ with square integrable weak derivatives.

\section{Analytic Preliminaries for Conformal Killing Vector Fields}

\subsection{Fundamental Identities for the Dual 1-Form}
Let $(M,g)$ be a closed $n$-dimensional smooth Riemannian manifold. A smooth vector field $X$ is called a conformal Killing vector field if its Lie derivative satisfies 
$$\mathcal{L}_Xg=h g$$
for some smooth function $h \in C^\infty(M)$. Contracting both sides with $g^{ij}$, and using the identity $\mathrm{tr}_{g}(\mathcal{L}_Xg)=2\operatorname{div} X$, we obtain
	\[h=\frac{2}{n}\operatorname{div} X.\]
Denote by $\theta=X^\flat$ the  dual $1$-form of the conformal Killing vector field $X$. In this section, our goal is to derive several fundamental identities involving $\theta$, which will serve as the analytic basis for the subsequent Bochner-type estimates. 
	
We begin with the Bochner-Weitzenböck formula for the rough Laplacian acting on $\theta$. It is well known (cf. \cite[Section 2]{Semm}) that for the dual form of a conformal Killing vector field, the following identity holds: 
 \begin{align} 
 \nabla^*\nabla\theta&=\mathrm{Ric}(\theta)+\frac{n-2}{2}dh\nonumber\\
 &=\mathrm{Ric}(\theta)+\frac{n-2}{n}d(\operatorname{div} X)\nonumber\\
  &=\mathrm{Ric}(\theta)-\frac{n-2}{n}d\delta \theta.\label{mainKillForm1}
 \end{align}
Recall also the classical Bochner formula for the Hodge Laplacian $\Delta_d=d\delta+\delta d$: 
$$\Delta_d\theta=\nabla^*\nabla\theta+\mathrm{Ric}(\theta).$$
Substituting (\ref{mainKillForm1}) into the Bochner formula, we obtain
\[d\delta\theta+\delta d\theta=2\mathrm{Ric}(\theta)-\frac{n-2}{n}d\delta\theta.\]
After rearranging terms, this gives the following useful relation:
\begin{equation}\label{mainKillForm2}
\frac{2n-2}{n}d\delta \theta+\delta d\theta=2\mathrm{Ric}(\theta).
\end{equation}
These identities allow us to derive a global integral estimate for $\theta$.
 \begin{lemma}\label{K3}
Let $(M,g)$ be a closed $n$-dimensional smooth Riemannian manifold, and let $\theta$ be the $1$-form dual to a conformal Killing vector field. Then the following integral inequality holds:
 $$\int_M |\nabla \theta|^2\leq \int_M\left\langle\mathrm{Ric}(\theta),\theta\right\rangle.$$
 \end{lemma}
 \begin{proof}
Taking the inner product of \eqref{mainKillForm1} with $\theta$ and  integrating over $M$,  we get
\begin{align*}
\int_M |\nabla \theta|^2= \int_M\left\langle\mathrm{Ric}(\theta),\theta\right\rangle-\frac{n-2}{n}\int_M |\delta\theta|^2.
\end{align*}
When $n\geq2$, the second term on the right-hand side is non-positive (and vanishes identically when $n=2$), which proves the desired inequality.
 \end{proof}
Next, we recall a pointwise structural identity for $\na\theta$, which can be found in Semmelmann \cite{Semm}.
\begin{lemma}[\cite{{Semm}}]\label{K4}
 Let $(M,g)$ be a closed $n$-dimensional smooth Riemannian manifold, and let $\theta$ be the $1$-form dual to a conformal Killing vector field $X$. Then, for any vector field $Y\in \mathfrak{X}(M)$, the following pointwise equalities hold:
 \begin{equation}\label{nablaKill}
 \nabla_Y\theta=\frac{1}{2}\iota_{Y}d\theta-\frac{1}{n}\delta\theta\cdot\omega_Y,
 \end{equation}
 and consequently,
 $$|\nabla \theta|^2=\frac{1}{2}|d\theta|^2+\frac{1}{n}(\delta\theta)^2,$$
where $\omega_Y=Y^\flat$  denotes the dual $1$-form of $Y$.
 \end{lemma}
 \begin{proof}
For arbitrary vector fields $Y,Z\in \mathfrak{X}(M)$, the exterior derivative and the Lie derivative satisfy the pointwise relations
 	\begin{equation*}
 	\begin{split}
 	&d\theta(Y,Z)=\langle\na_{Y}X,Z\rangle-\langle\na_{Z}X,Y\rangle,\\
 	&\mathcal{L}_Xg(Y,Z)=\langle\na_{Y}X,Z\rangle+\langle\na_{Z}X,Y\rangle.\\
 	\end{split}
 	\end{equation*}
Adding these two identities gives
 $$d\theta(Y,Z)+\mathcal{L}_Xg(Y,Z)=2\langle\na_{Y}X,Z\rangle.$$
 Since $\mathcal{L}_Xg=hg$ and $h=\frac{2}{n}\mathrm{div}X=-\frac{2}{n}\delta\theta$, it follows that
 \begin{align*}
 (\nabla_Y\theta)(Z)=&\langle\nabla_YX,Z\rangle\\
 =&\frac{1}{2}d\theta(Y,Z)+\frac{1}{2}\mathcal{L}_Xg(Y,Z)\\
 =&\frac{1}{2}\iota_{Y}d\theta(Z)-\frac{1}{n}\delta \theta\cdot \langle Y,Z \rangle,
 \end{align*}
which is exactly \eqref{nablaKill}. To obtain the second identity, we choose a local orthonormal frame $\{e_i\}_{i=1}^n$. Then, using \eqref{nablaKill}, we compute
\begin{align*}
|\nabla \theta|^2=&\sum_{i=1}^n\langle \nabla_{e_i}\theta,\nabla_{e_i}\theta\rangle=\sum_{i,j=1}^n |\left(\nabla_{e_i}\theta\right)(e_j)|^2\\
=&\sum_{i,j=1}^n\left[\frac{1}{2}d\theta(e_i,e_j)-\frac{1}{n}\delta\theta\cdot \langle e_i,e_j\rangle\right]^2\\
=&\sum_{\substack{i,j=1\\i\neq j}}^n\left(\frac{1}{2}d\theta(e_i,e_j)\right)^2+\sum_{i=1}^{n}\left(\frac{1}{n}\delta\theta\cdot \langle e_i,e_i\rangle\right)^{2}\\  
=&\frac{1}{2}|d\theta|^2+\frac{1}{n}(\delta\theta)^2,
\end{align*}
where the last equality uses $d\theta(e_i,e_i)=0$ and
\begin{equation*}
|d\theta|^{2}=\sum_{1\leq i<j\leq n}\left(d\theta(e_{i},e_{j})\right)^{2}=\sum_{i,j=1}^{n}\frac{1}{2}\left(d\theta(e_{i},e_{j})\right)^{2}.
\end{equation*}
 \end{proof}
As a consequence of Lemma \ref{K4}, we obtain the following inner product identity involving the Hessian of $\delta\theta$.
  \begin{lemma}\label{K5}
 Let $(M,g)$ be a closed $n$-dimensional smooth Riemannian manifold, and let $\theta$ be the $1$-form dual to a conformal Killing vector field.  The following pointwise identity holds:
 $$\langle \mathrm{Hess}\, \delta \theta,\nabla \theta\rangle=\frac{1}{n}\delta \theta\cdot\Delta_d\delta \theta.$$
 \end{lemma}
 \begin{proof}
 Fix a point $p\in M$ and choose a local orthonormal frame  $\{e_i\}_{i=1}^n$ in a neighborhood of $p$ such that $\left.\nabla_{e_i}e_j\right|_p=0$ for all $1\leq i,j\leq n$. At $p$,  Lemma \ref{K4} gives
 \begin{align*}
\langle \mathrm{Hess}\, \delta \theta,\nabla \theta\rangle
=&\sum_{i,j=1}^n \mathrm{Hess}\, \delta \theta(e_i,e_j)\cdot\left(\nabla_{e_i} \theta\right)(e_j)\\
=&\sum_{i,j=1}^n \mathrm{Hess}\, \delta \theta(e_i,e_j)\cdot\left[\frac{1}{2}d\theta(e_i,e_j)-\frac{1}{n}\langle e_i,e_j\rangle\cdot\delta \theta\right]\\
=&-\frac{1}{n}\sum_{i=1}^n \mathrm{Hess}\, \delta \theta(e_i,e_i)\cdot\delta \theta\\
=&\frac{1}{n}\delta \theta\cdot\Delta_d\delta \theta,
\end{align*}
where the second last equality follows from the fact that $\mathrm{Hess}\, \delta \theta$ is symmetric while $d\theta$ is skew-symmetric.
 \end{proof}
Finally, we derive a differential inequality for $|\theta|$, which will be essential for the Moser iteration in later sections.
\begin{lemma}\label{K1}
Let $(M,g)$ be a closed $n$-dimensional smooth Riemannian manifold, and let $\theta$ be the $1$-form dual to a conformal Killing vector field. Then the following pointwise inequality holds:
$$|\theta| \cdot \Delta_d|\theta| \leq \left\langle\mathrm{Ric}(\theta)-\frac{n-2}{n}d\delta\theta,\theta\right\rangle.$$
\end{lemma}
\begin{proof}
We first recall the  identity 	
	\[\frac{1}{2}\Delta_d|\theta|^{2}= -|\nabla\theta|^{2}+\langle\nabla^{*}\nabla\theta,\theta\rangle.\]
On the other hand, the Hodge Laplacian acting on the function $|\theta|^2$ yields:
	\[
	\frac{1}{2}\Delta_d|\theta|^{2} = |\theta| \cdot \Delta_d|\theta| - \left|\nabla|\theta|\right|^{2}.
	\]
Combining these two relations, using the Kato inequality $\left|\nabla|\theta|\right|\leq |\nabla\theta|$ and (\ref{mainKillForm1}), we obtain
\begin{align*}
	|\theta| \cdot \Delta_d|\theta| =&\left|\nabla|\theta|\right|^{2}-|\nabla\theta|^{2}+\langle\nabla^{*}\nabla\theta,\theta\rangle\\
	\leq&\langle\nabla^{*}\nabla\theta,\theta\rangle\\
	= & \left\langle \mathrm{Ric}(\theta)-\frac{n-2}{n}d\delta\theta,\theta\right\rangle.
	\end{align*}	
\end{proof}

\subsection{ A Higher-Order Bochner Identity}
In this subsection, we develop a refined Bochner-type identity for the  1-form $\theta$ dual to a conformal Killing vector field. More precisely, instead of applying the classical Bochner formula to $|\theta|^{2}$, we apply it to $|\na\theta|^{2}$, which yields a higher-order estimate essential for the Moser iteration in Section \ref{Sec}. To this end, we shall work in the framework of $T^*M$-valued differential forms.

Let  $\Lambda^p(T^*M) \otimes T^*M$ denote the bundle of $T^*M$-valued $p$-forms. Its sections are smooth tensor fields that take values in the cotangent bundle. We refer to \cite[Section 2]{PS} for a comprehensive treatment in the more general setting of forms valued in an arbitrary vector bundle; here we recall only the necessary definitions.

Let $\nabla$ be the connection on $\Lambda^p(T^*M) \otimes T^*M$  induced by the Levi-Civita connection of $(M,g)$. For any section $T$ of this bundle, the second covariant derivative is defined by
$$
\nabla_{X, Y}^2 T=\nabla_X \nabla_Y T-\nabla_{\nabla_X Y} T,\quad X,Y\in \mathfrak{X}(M).$$ 
Choosing a local orthonormal frame $\{e_i\}_{i=1}^n$, the rough Laplacian (or connection Laplacian) $\overline{\Delta}=\nabla^* \nabla$ is given by
$$\nabla^* \nabla T = -\sum_{i=1}^n \nabla_{e_i, e_i}^2 T.$$
The curvature operator $R(X,Y)$ acting on $T^*M$-valued $p$-forms is defined as the commutator of second covariant derivatives:
$$R(X, Y) T = \nabla_{X, Y}^2 T - \nabla_{Y, X}^2 T.$$
 We also recall the exterior covariant derivative $d^\nabla$, which maps $T^*M$-valued $p$-forms  to $T^*M$-valued $(p+1)$-forms. For $X_0, \ldots, X_p \in\mathfrak{X}(M)$,  it is defined by
\begin{equation}\label{defofDnabla}
(d^\nabla T)\left(X_0, \ldots, X_p\right) = \sum_{i=0}^p (-1)^i \left(\nabla_{X_i} T\right)\left(X_0, \ldots, \hat{X}_i, \ldots, X_p\right).
\end{equation}
The codifferential $\delta^\nabla$ is the formal $L^{2}$-adjoint of $d^\nabla$, acting on a $T^*M$-valued $(p+1)$-form $S$ by negative contraction (cf. \cite[Page 76]{PS}):
\begin{equation}\label{defofdeltanabla}
(\delta^\nabla S)(X_1,\ldots, X_p) = -\sum_{i=1}^n (\nabla_{e_i}S)(e_i, X_1,\ldots, X_p).
\end{equation}
For the case $p=1$, i.e., for $T^{\ast}M$-valued $1$-forms, the generalized Bochner-Weitzenböck formula (cf. \cite[(2.8)]{PS}) states that
$$ \left(d^\nabla \delta^\nabla + \delta^\nabla d^\nabla\right) T = \nabla^* \nabla T + \mathrm{Ric}(T),$$
where $\mathrm{Ric}(T)$ is a $T^*M$-valued $1$-form. In a local orthonormal frame $\{e_i\}_{i=1}^n$, it is given explicitly by
\begin{equation}\label{Ricdef}
\mathrm{Ric}(T)(Y) = \sum_{i=1}^n \left(R(e_i, Y)T\right)(e_i)= \sum_{i=1}^n \left[R(e_i, Y)\left(T(e_i)\right)-T\left(R(e_i, Y)e_i\right)\right]
\end{equation}
for any $Y\in \mathfrak{X}(M)$.

We now apply this machinery to the $T^{\ast}M$-valued $1$-form $\na\theta$, where $\theta$ is the $1$-form dual to a conformal Killing vector field. Following the computation in \cite[Page 81]{PS}, we derive the pointwise identity below for $\Delta_d(|\nabla \theta|^2)$.

\begin{lemma}\label{K2}
	Let $(M,g)$ be a closed $n$-dimensional smooth Riemannian manifold, and let $\theta$ be the $1$-form dual to a conformal Killing vector field. Then the following pointwise identity holds:
	\begin{align}
	\frac{1}{2}\Delta_d(|\nabla \theta|^2)+\left|\nabla^2 \theta\right|^2\nonumber 
	=& \left\langle\nabla\mathrm{Ric}(\theta),\nabla \theta\right\rangle-\frac{n-2}{n(n-1)}\delta \theta\cdot\delta\left[\mathrm{Ric}(\theta)\right]\\
	&+\left\langle \nabla^*\left[R(\cdot,\cdot)\theta\right],\nabla \theta\right\rangle-\left\langle\mathrm{Ric}(\nabla \theta),\nabla \theta\right\rangle,\label{InequalityforGra}
	\end{align}
where the local coordinate expressions for  $\nabla^*[R(\cdot,\cdot)\theta]$ and $\langle \mathrm{Ric}(\nabla \theta), \nabla \theta \rangle$ are given in \eqref{Rtheta} and \eqref{Rictheta}, respectively.
\end{lemma}
\begin{proof}
We begin with the standard pointwise identity
 \begin{align*}
 \frac{1}{2}\Delta_d(|\nabla \theta|^2)
=&\langle \nabla^*\nabla \nabla \theta,\nabla \theta\rangle-\left|\nabla^2 \theta\right|^2,
\end{align*}
Thus, it suffices to compute the term $\langle \nabla^*\nabla \nabla \theta,\nabla \theta\rangle$. To this end, we view $\na\theta$ as a $T^{\ast}M$-values $1$-form. The generalized Bochner-Weitzenböck formula gives 
\begin{equation}\label{K2111}
\langle \nabla^*\nabla \nabla \theta,\nabla \theta\rangle=\langle d^\nabla\delta^\nabla \nabla \theta,\nabla \theta\rangle+\langle \delta^\nabla d^\nabla\nabla \theta,\nabla \theta\rangle-\langle \mathrm{Ric}(\nabla \theta),\nabla \theta\rangle.
\end{equation}
We now evaluate the two terms on the right-hand side separately.\\
\textbf{First term} $\langle d^\nabla\delta^\nabla \nabla \theta,\nabla \theta\rangle$. When the operator $d^\nabla$ acts, the $1$-forms $\mathrm{Ric}(\theta)$ and $d\delta \theta$ are viewed as  $T^*M$-valued $0$-forms. Recalling the definition of  $d^\nabla$ in $\eqref{defofDnabla}$ and the equation \eqref{mainKillForm1}, it follows that 
$$d^\nabla\delta^\nabla \nabla \theta=d^\nabla(\nabla^*\nabla \theta)=d^\nabla\left[\mathrm{Ric}(\theta)-\frac{n-2}{n}d\delta \theta\right]=\nabla\mathrm{Ric}(\theta)-\frac{n-2}{n}d^\nabla d\delta \theta.$$ 
Note that $d^\nabla d\delta \theta=\nabla d\delta \theta=\mathrm{Hess}\, \delta \theta$. 
From Lemma \ref{K5}, we arrive at
\begin{align*}
\langle d^\nabla d\delta \theta,\nabla \theta\rangle=&\langle \mathrm{Hess}\, \delta \theta,\nabla \theta\rangle
=\frac{1}{n}\delta \theta\cdot\Delta_d\delta \theta.
\end{align*}
On the other hand, applying the codifferential $\delta$ to both sides of the identity \eqref{mainKillForm2} yields
 $$\frac{n-1}{n} \Delta_d(\delta \theta)=\delta\left[\mathrm{Ric}(\theta)\right].$$
Combining these, we get
\begin{align}
\langle d^\nabla\delta^\nabla \nabla \theta,\nabla \theta\rangle=&\langle \nabla\mathrm{Ric}(\theta),\nabla \theta\rangle-\frac{n-2}{n^2}\delta \theta\cdot\Delta_d\delta \theta\nonumber\\
=&\langle \nabla\mathrm{Ric}(\theta),\nabla \theta\rangle-\frac{n-2}{n(n-1)}\delta \theta\cdot\delta\left[\mathrm{Ric}(\theta)\right].\label{K2222}
\end{align}
\textbf{Second term} $\langle \delta^\nabla d^\nabla\nabla \theta,\nabla \theta\rangle$. For arbitrary vector fields $Y,Z\in \mathfrak{X}(M)$, a direct computation gives 
\begin{align*}
d^\nabla\nabla \theta(Y,Z)=&\left[\nabla_{Y}(\nabla \theta)\right](Z)- \left[\nabla_{Z}(\nabla \theta)\right](Y)\\
=&\nabla_{Y}\nabla_Z \theta-\nabla_{\nabla_YZ}\theta-\nabla_{Z}\nabla_Y \theta+\nabla_{\nabla_ZY}\theta\\
=&R(Y,Z)\theta.
\end{align*}
Thus, $R(\cdot,\cdot)\theta$, which may be viewed as a $T^*M$-valued $2$-form. Applying the codifferential $\delta^\nabla$ as defined in \eqref{defofdeltanabla}, and noting that $\de^{\na}$  coincides with the standard divergence operator $\na^{\ast}$ on such forms (cf. \cite[Page 76]{PS}), we obtain
\begin{equation}\label{K2333}
\langle \delta^\nabla d^\nabla\nabla \theta,\nabla \theta\rangle=\langle \nabla^*\left[R(\cdot,\cdot)\theta\right],\nabla \theta\rangle.
\end{equation}
Substituting \eqref{K2222} and \eqref{K2333} into \eqref{K2111} yields the \eqref{InequalityforGra}. It remains only to provide the local coordinate expressions for the curvature terms, which we do below for completeness.

Fix a point $p\in M$ and choose a local orthonormal frame $\{e_i\}_{i=1}^n$ in a neighborhood of $p$ such that $\left.\nabla_{e_i}e_j\right|_p=0$ for all $1\leq i,j\leq n$. At $p$, the term $\nabla^*[R(\cdot,\cdot)\theta]$ in \eqref{InequalityforGra}  is given by
\begin{align}
\nabla^*\left[R(\cdot,\cdot)\theta\right](Y)
=&-\sum_{i=1}^{n}(\na_{e_{i}}\left[R(\cdot,\cdot)\theta\right] )(e_{i},Y) \nonumber \\
=&\sum_{i=1}^n-\nabla_{e_i}\left[R(e_i,Y)\theta\right]+R(\nabla_{e_i}e_i,Y)\theta+R(e_i,\nabla_{e_i}Y)\theta\nonumber \\
=&\sum_{i=1}^n-\nabla_{e_i}\left[R(e_i,Y)\theta\right]+R(e_i,\nabla_{e_i}Y)\theta,\label{Rtheta}
\end{align}
for any vector field $Y\in \mathfrak{X}(M)$. Similarly, by \eqref{Ricdef}, the term $\left\langle\mathrm{Ric}(\nabla \theta),\nabla \theta\right\rangle$ in \eqref{InequalityforGra} can be expressed as
\begin{align}
\left\langle\mathrm{Ric}(\nabla \theta),\nabla \theta\right\rangle=&\sum_{j=1}^n\left\langle\left[\mathrm{Ric}(\nabla \theta)\right](e_j),\nabla_{e_j} \theta\right\rangle\nonumber\\
=&\sum_{i,j=1}^n\left\langle R(e_i,e_j)(\nabla_{e_i} \theta)-\nabla_{R(e_i,e_j)e_i}\theta,\nabla_{e_j} \theta\right\rangle\nonumber\\
=&\sum_{i,j=1}^n\left\langle R(e_i,e_j)(\nabla_{e_i} \theta),\nabla_{e_j} \theta\right\rangle+\sum_{i,j=1}^n\left\langle\nabla_{\mathrm{Ric}(e_j)}\theta,\nabla_{e_j} \theta\right\rangle.\label{Rictheta}
\end{align}
This completes the proof. 
\end{proof}

\begin{remark}
We point out that
$$\langle d^\nabla\delta^\nabla \nabla \theta,\nabla \theta\rangle=\langle \nabla\nabla^* \nabla \theta,\nabla \theta\rangle.$$
Combining this with  (\ref{K2222}), we obtain the identity
$$\langle \nabla\mathrm{Ric}(\theta),\nabla \theta\rangle-\frac{n-2}{n(n-1)}\delta \theta\cdot\delta\left[\mathrm{Ric}(\theta)\right]=\langle \nabla\nabla^* \nabla \theta,\nabla \theta\rangle.$$
Consequently, for any general $1$-form $\theta$, the formula \eqref{InequalityforGra} reduces to
$$
\frac{1}{2}\Delta_d(|\nabla \theta|^2)+\left|\nabla^2 \theta\right|^2
=\langle \nabla\nabla^* \nabla \theta,\nabla \theta\rangle+\left\langle \nabla^*\left[R(\cdot,\cdot)\theta\right],\nabla \theta\right\rangle-\left\langle\mathrm{Ric}(\nabla \theta),\nabla \theta\right\rangle.
$$
Using the local coordinate expressions for the terms $\left\langle \nabla^*\left[R(\cdot,\cdot)\theta\right],\nabla \theta\right\rangle$ and $\left\langle\mathrm{Ric}(\nabla \theta),\nabla \theta\right\rangle$ in Lemma \ref{K2}, the above formula coincides with that in  \cite{Aub,ACGR,LR}. 
 \end{remark}

\section{Estimates for the Dual $1$-Form}\label{Sec}
\subsection{$L^{\infty}$ Estimate for the Dual $1$-form}
In this subsection, we establish an $L^{\infty}$ estimate for the  $1$-form $\theta$ dual to a conformal Killing vector field. The argument is a Moser iteration procedure, adapted from  \cite[Lemma 4.1]{PS}. 

We assume throughout that the manifold satisfies the following Sobolev inequality:
\begin{equation} \label{Sobolev32}
\|f\|_{\frac{2n}{n-2}} \leq \|f\|_{2} + C_{s}\|df\|_{2}, \quad \text{for all } f \in H^1(M).
\end{equation}
where $C_{s}>0$ is the Sobolev constant.

\begin{lemma}\label{L4}
	Let $(M,g)$ be a closed $n$-dimensional smooth Riemannian manifold satisfying the Sobolev inequality \eqref{Sobolev32}. Suppose that $\theta$ is the $1$-form dual to a conformal Killing vector field $X$, and assume that the Ricci curvature satisfies 
	$$\mathrm{Ric}<\la,$$
	for some positive constant $\la$. Then
	\begin{equation}\label{E1}
	\|\theta\|_{\infty} \leq \exp\left\{C(n)\sqrt{\lambda}C_{s}\right\} \|\theta\|_{2}.
	\end{equation}
\end{lemma}

\begin{proof}
Let $f=|\theta|$. Since $\mathrm{Ric}<\la$, for $k>0$, we have
\begin{align}
\int_M|df^k|^2=&k^2\int_Mf^{2k-2}|\nabla f|^2 \nonumber\\
=&\frac{k^2}{2k-1}\int_Mf^{2k-1}\Delta_d f \nonumber\\
\leq&\frac{k^2}{2k-1}\int_M|\theta|^{2k-2}\langle \mathrm{Ric}(\theta)-\frac{n-2}{n}d\delta\theta,\theta\rangle\nonumber\\
\leq &\frac{k^2}{2k-1}\int_M\lambda f^{2k}-\frac{k^2(n-2)}{n(2k-1)}\int_M |\theta|^{2k-2}\left\langle d\delta\theta,\theta\right\rangle.\label{SobA}
\end{align}
It remains to handle the last term. Note that
$$\int_M |\theta|^{2k-2} \langle d\delta \theta, \theta \rangle=\int_M  \langle d\delta \theta, |\theta|^{2k-2}\theta \rangle = \int_M(\delta \theta)\delta (|\theta|^{2k-2} \theta),$$
where $\delta \theta$ and $\delta (|\theta|^{2k-2} \theta)$ are two functions.

Now, since $\theta$ is the dual form of the vector field $X$, we have $|\theta|^2=|X|^2$. Observe that 
\begin{align*}
\langle d|\theta|^{2k-2}, \theta \rangle&=2(k-1)|\theta|^{2k-3}\langle d|\theta|,\theta\rangle\\
&=2(k-1)|\theta|^{2k-3}\na_{X}|\theta|\\
&=2(k-1)|\theta|^{2k-3}\na_{X}|X|\\
&=(k-1)|\theta|^{2k-4}\na_{X}|X|^{2}\\
&=(k-1)|\theta|^{2k-4}\na_{X}\langle X,X\rangle\\
&=(k-1)|\theta|^{2k-4}2\langle\na_{X}X,X\rangle.\\
\end{align*}
Hence,
\begin{align*}
\delta (|\theta|^{2k-2} \theta) =& |\theta|^{2k-2} \delta \theta - \langle d|\theta|^{2k-2}, \theta \rangle\\
=& |\theta|^{2k-2} \delta \theta - (k-1)|\theta|^{2k-4} 2\langle \nabla_X X,X\rangle\\
=& \left( 1 + \frac{2k-2}{n} \right) |\theta|^{2k-2} \delta \theta,
\end{align*}
where the last equality is due to $X$ is a conformal Killing vector field and so satisfies 
$$2\langle \nabla_X X,X\rangle=\mathcal{L}_Xg(X,X)=h\langle X,X\rangle=\frac{2}{n}\mathrm{div}\, X |X|^2=-\frac{2}{n}(\delta \theta)|\theta|^2.$$
Substituting this back, we obtain
\begin{align*}
 \int_M |\theta|^{2k-2} \langle d\delta \theta, \theta \rangle=\int_M(\delta \theta)\delta (|\theta|^{2k-2} \theta) =& \left( 1 + \frac{2k-2}{n} \right) \int_M |\theta|^{2k-2} |\delta \theta|^2
 \end{align*}
 is a positive term. Thus, the last term in (\ref{SobA}) is non-positive, and we can drop it to obtain the simpler estimate
 $$\int_M|df^k|^2\leq\lambda\frac{k^2}{2k-1}\int_M f^{2k}.$$
Then the result will follow from a standard Moser iteration, for instance \cite[Theorem 9.2.7]{Pet}. For completeness, we include the details here. Substitute $f^k$ into the Sobolev inequality, we obtain 
	$$
	\|f^k\|_{\frac{2n}{n-2}} \leq \|f^k\|_{2} + C_{s}\|df^k\|_{2}\leq \left(1+ C_{s}k\left(\frac{\lambda}{2k-1}\right)^{\frac{1}{2}}\right)\|f^k\|_{2},
	$$
and so
	$$\|f\|_{\frac{2nk}{n-2}} \leq  \left(1+ C_{s}k\left(\frac{\lambda}{2k-1}\right)^{\frac{1}{2}}\right)^{\frac{1}{k}}\|f\|_{2k}.
	$$
Letting $k=\nu^i=\left(\frac{n}{n-2}\right)^i$ gives
	\begin{align*}
	\|f\|_{2\nu^{i+1}} \leq & \left[1+ C_{s}\nu^i\left(\frac{\lambda}{2\nu^i-1}\right)^{\frac{1}{2}}\right]^{\frac{1}{\nu^i}}\|f\|_{2\nu^i}\leq \left(1+ C_{s}\sqrt{\lambda\nu^i}\right)^{\frac{1}{\nu^i}}\|f\|_{2\nu^i}
	\end{align*}
	Applying Theorem \ref{B1} with
 $$ \nu = \frac{n}{n-2},\quad A(i)=\|f\|_{2\nu^i},\quad r=C_s\sqrt{\lambda},$$
 we obtain
\begin{align*}
\|\theta\|_{\infty}&=\|f\|_{\infty}=\lim_{i\to+\infty}\|f\|_{2\nu^i}\\
&\leq \exp\left(\frac{C_s\sqrt{\lambda}\sqrt{\nu}}{\sqrt{\nu}-1}\right)\|f\|_2\\
&\leq  \exp\left\{C(n)\sqrt{\lambda}C_{s}\right\} \|\theta\|_{2}.
\end{align*}
where $C(n)=\frac{\sqrt{\nu}}{\sqrt{\nu}-1}=\frac{\sqrt{n}}{\sqrt{n}-\sqrt{n-2}}$. This completes the proof. 
\end{proof}
\subsection{$L^{\infty}$ Estimate for the Gradient of Dual $1$-form}
In this subsection, we establish an $L^{\infty}$ estimate for $|\nabla\theta|$, where $\theta$ is the  $1$-form dual to  a conformal Killing vector field. This estimate is significantly more delicate than the previous one for $\theta$, as it requires controlling the second-order derivatives of $\theta$ via the curvature terms. The argument follows the Moser iteration strategy developed by Aubry \cite[Proposition 4.2]{Aub} (see also \cite{ACGR,LR}), with the key difference lying in the derivation of the integral inequality in Lemma \ref{L7} below. Unlike the approach in \cite{PS}, our estimate does not require bounds on the covariant derivative $\na R$ of the curvature tensor.
\begin{lemma}\label{L7}
Let $(M,g)$ be a closed $n$-dimensional smooth Riemannian manifold, and let $\theta$ be the $1$-form dual to a conformal Killing vector field. Set $u=|\nabla\theta|$. Then for any $k>1$, the following integral inequality holds:
\begin{align*}
\frac{1}{2}\int_M\Delta_d(u^2)u^{2k-2}+\int_M\left|\nabla^2 \theta\right|^2u^{2k-2}
\leq &9k \int_M |R\theta|^2 \cdot u^{2(k-1)} +20nk\int_M |\mathrm{Ric}(\theta)|^2 \cdot u^{2(k-1)}\\
&+2\int_M \underline{\mathrm{Ric}}^{-} u^{2k}+2\int_M |R| u^{2k}+\frac{ 3(k-1)}{4} \int_M |\nabla u|^2 u^{2(k-1)}.
\end{align*}
\end{lemma}
	\begin{remark}
			The proof of Lemma \ref{L7} is technically involved and relies on the Bochner-type identity \eqref{InequalityforGra} established in Lemma \ref{K2}. The detailed derivation is deferred to Appendix \ref{appendixA}, where all auxiliary integral identities are provided.
	\end{remark}
With Lemma \ref{L7} at hand, we can now prove the main $L^{\infty}$ estimate for $\na\theta$. The following proposition is the central result of this subsection.

\begin{proposition}\label{P6}
Let $(M,g)$ be a closed $n$-dimensional smooth Riemannian manifold satisfying the Sobolev inequality \eqref{Sobolev32}. Assume that there exist positive constants $\kappa,\lambda, K,p$ with $p>n/2$ such that 
	\[
	-(n-1)\kappa\leq \operatorname{Ric}\leq \lambda,\quad 
	\|R\|_{2p} \leq K,\quad\text{and}\quad 
	\operatorname{diam}(M) \leq D.
	\]
Let $\theta$ be the $1$-form dual of a conformal Killing vector field. Then there exists a constant $C(n,p) > 0$ such that	
	\begin{align}
	D \|\nabla\theta\|_\infty
	&\le \max \Bigg\{
	C(n,p)(1 + \sqrt{t})^{\alpha} \sqrt{\lambda D^{2}} \|\theta\|_{2}, \nonumber\\
	&\qquad C(n,p)(1 + \sqrt{t})^{\beta} (\sqrt{\lambda D^{2}})^{\frac{\beta}{\alpha}}
	\exp\left[ C(n,p) \sqrt{\lambda} C_{s} \right] \|\theta\|_{2}
	\Bigg\},\label{E2}
	\end{align}
where 
$$\a=\frac{2pn}{2p-n},\quad\b=\frac{2pn}{2p-n+pn},\quad t = 4 C_s\sqrt{B}\sqrt{1+BD^{2}}\quad and\quad B=5n(\kappa+\lambda+K).$$
\end{proposition}
\begin{proof}
Let $u = |\nabla \theta|$. First, by Kato's inequality $|\na u|\leq |\na^{2}\theta|$. Consequently,
\begin{align*}
u \Delta_d u=&|\nabla\theta| \cdot \Delta_d|\nabla\theta|=\frac{1}{2}\Delta_d|\nabla\theta|^{2} + \left|\nabla|\nabla\theta|\right|^{2}\\
\leq& \frac{1}{2}\Delta_d\left(|\nabla\theta|^{2}\right)+ \left|\nabla^2\theta\right|^{2}.
\end{align*}
Hence, from Lemma \ref{L7}, we can get for any $k>1$,
	\begin{align*}
	\int_M |d(u^k)|^2 =& k^2 \int_M |\nabla u|^2 u^{2(k-1)} = \frac{k^2}{2k-1} \int_M \langle u^{2k-1}, \Delta_d u \rangle \\
	\leq& \frac{k^2}{2k-1} \int_M \left( \frac{1}{2}\Delta_d\left(|\nabla\theta|^{2}\right)+ \left|\nabla^2\theta\right|^{2} \right) u^{2(k-1)} \\
	\leq & \frac{k^2}{2k-1} \left( 2\int_M \underline{\mathrm{Ric}}^{-}  u^{2k} +2\int_M |R| u^{2k}+20nk\int_M|\mathrm{Ric}(\theta)|^2u^{2(k-1)} \right. \\
	& + \left. 9k \int_M |R\theta|^2 \cdot u^{2(k-1)} +\frac{ 3(k-1)}{4} \int_M |\nabla u|^2 u^{2(k-1)} \right).
	\end{align*}	
Rearranging the term involving  $|\na u|^{2}$ to the left-hand side, we obtain	
	\begin{align}
	\int_M |d(u^k)|^2 &\leq\frac{4k^2}{5k-1} \left(2\int_M  \underline{\mathrm{Ric}}^{-}u^{2k} +2\int_M |R| u^{2k}+ 20nk \int_M\left( |R\theta|^2+|\mathrm{Ric}(\theta)|^2\right) \cdot u^{2(k-1)}\right)\nonumber\\
	&\leq  2k \int_M \left(\underline{\mathrm{Ric}}^{-}+|R|\right)  u^{2k} + 20nk^2 \int_M\left( |R\theta|^2+|\mathrm{Ric}(\theta)|^2\right) \cdot u^{2(k-1)}.\label{E3}
	\end{align}
	Here we use the inequality $\frac{4k^{2}}{5k-1}\leq k$ for $k\geq1$.

We now estimate the right-hand side of (\ref{E3}) using Hölder's inequality. For the first term, we have
	\[ \int_M \left(\underline{\mathrm{Ric}}^{-}+|R|\right)u^{2k}\leq\left(\|\underline{\mathrm{Ric}}^{-} \|_p+\|R\|_p  \right)\|u\|^{2k}_{\frac{2kp}{p-1}}.  \]
For the second term, we apply Hölder's inequality with exponents $p$ and $\frac{p}{p-1}$,
\[\int_M\left( |R\theta|^2+|\mathrm{Ric}(\theta)|^2\right) \cdot u^{2(k-1)}\leq (\|R\|^2_{2p}+\|\mathrm{Ric}\|_{2p}^2) \|\theta\|^2_\infty \|u\|^{2(k-1)}_{\frac{2(k-1)p}{p-1}} \]	
and similarly for the Ricci term. Since $-(n-1)\kappa\leq\mathrm{Ric}\leq \la$, we have
$$
\|\underline{\mathrm{Ric}}^{-}\|_p +\|R\|_p\leq (n-1) \kappa+K\leq B, $$
and
$$20n(\|\mathrm{Ric}\|_{2p}^2+\|R\|^2_{2p})\leq 20n(\lambda^2+K^2) \leq B^2,$$	
where $B:=5n(\kappa+\lambda+K)$. Substituting these estimates into (\ref{E3}), we obtain	
\begin{equation}\label{finalforduk}
\int_M |d(u^k)|^2\leq 4k^{2}B\|u\|^{2k}_{\frac{2kp}{p-1}} + 4k^{2}B^{2}\|\theta\|^2_\infty \|u\|^{2(k-1)}_{\frac{2(k-1)p}{p-1}}.
\end{equation}
Substitute $u^k$ into the Sobolev inequality \eqref{Sobolev32}. Based on the estimate \eqref{finalforduk} for $\int_X |d(u^k)|^2$,  using the interpolation inequality
	\[
	\|u\|_{2k} \leq \|u\|_{\frac{2kp}{p-1}} \leq \|u\|^{\frac{1}{k}}_\infty \|u\|^{1-\frac{1}{k}}_{\frac{2(k-1)p}{p-1}},
	\]
 we conclude that 
	\begin{align*}
	\|u\|^k_{\frac{2nk}{n-2}} &= \|u^k\|_{\frac{2n}{n-2}} \\
	&\leq \|u^k\|_2 + C_s \|d(u^k)\|_2 \\
	&\leq \|u^k\|_2 + 2C_s k  \sqrt{ B\|u\|^{2k}_{\frac{2kp}{p-1}} +B^{2} \|\theta\|^2_\infty \|u\|^{2(k-1)}_{\frac{2(k-1)p}{p-1}} } \\
	&\leq \left( \|u\|_\infty^{\frac{1}{k}} \|u\|^{1-\frac{1}{k}}_{\frac{2(k-1)p}{p-1}} \right)^k \\
	&\quad + 2C_s k \sqrt{ B\left( \|u\|_\infty^{\frac{1}{k}} \|u\|^{1-\frac{1}{k}}_{\frac{2(k-1)p}{p-1}} \right)^{2k} + B^{2}\|\theta\|^2_\infty \|u\|^{2(k-1)}_{\frac{2(k-1)p}{p-1}} } \\
	&= \left( \|u\|_\infty + 2C_s k \sqrt{ B\|u\|^2_\infty +B^{2} \|\theta\|^2_\infty } \right) \|u\|^{k-1}_{\frac{2(k-1)p}{p-1}}.
	\end{align*}
Dividing both sides by $\|u\|_\infty$ and raising both sides to the power of $\frac{2n}{n-2}$ yields the recurrence relation:
\begin{equation}\label{mainRecur}
	\left( \frac{\|u\|_{\frac{2nk}{n-2}}}{\|u\|_\infty} \right)^{\frac{2nk}{n-2}}
	\leq \left( 1 +2C_s k\sqrt{B}  \sqrt{ 1 + \frac{B \|\theta\|^2_\infty}{\|u\|^2_\infty} } \right)^{\frac{2n}{n-2}}
	\left( \frac{\|u\|_{\frac{2(k-1)p}{p-1}}}{\|u\|_\infty} \right)^{\gamma \frac{2(k-1)p}{p-1}},
	\end{equation}
	where $\gamma = \frac{n(p-1)}{p(n-2)} > 1$.
	
	Now, define the sequence $\{a_i\}_{i\geq 0}$ by setting $a_0 = \frac{2p}{p-1}$ and imposing the recursion formula
	\[
	a_{i+1} = \gamma a_i + \frac{2n}{n-2}.
	\]
Then, solving this recurrence relation gives
	\[
	a_i = \gamma^i (a_0 + \gamma_0) - \gamma_0,
	\]
	where $\gamma_0 = \frac{2n}{(n-2)(\gamma-1)} = \frac{2pn}{2p-n}$. Let $\{k_i\}$ be the sequence of powers defined by
	\[
	k_i = \frac{p-1}{2p} a_i + 1 = \gamma^{i+1} + \gamma^i - \gamma + 1 < 2\gamma^{i+1},
	\]
	which satisfies
	\[
	\frac{2n}{n-2} k_i = \frac{n(p-1)}{p(n-2)} a_i + \frac{2n}{n-2} = a_{i+1}.
	\]
	 Substituting the power $k=k_i$ into \eqref{mainRecur}, we can rewrite it as 
	\begin{align*}
	\left( \frac{\|u\|_{a_{i+1}}}{\|u\|_\infty} \right)^{\frac{a_{i+1}}{\gamma^{i+1}}}
	&\leq \left( 1 + 2C_s k_{i}\sqrt{B} \sqrt{ 1 + \frac{B \|\theta\|^2_\infty}{\|u\|^2_\infty} } \right)^{\frac{2n}{(n-2)\gamma^{i+1}}}
	\left( \frac{\|u\|_{a_i}}{\|u\|_\infty} \right)^{\frac{a_i}{\gamma^i}}\\
	&\leq \left( 1 + 4C_s \gamma^{i+1}\sqrt{B} \sqrt{ 1 + \frac{B \|\theta\|^2_\infty}{\|u\|^2_\infty} } \right)^{\frac{2n}{(n-2)\gamma^{i+1}}}
	\left( \frac{\|u\|_{a_i}}{\|u\|_\infty} \right)^{\frac{a_i}{\gamma^i}}.
	\end{align*}
Applying Theorem \ref{B2} with
 $$ \gamma = \frac{n(p-1)}{p(n-2)},\quad A(i)=\left( \frac{\|u\|_{a_i}}{\|u\|_\infty} \right)^{\frac{a_i}{\gamma^i}},\quad s=\frac{2n}{n-2},\quad r=4C_s\sqrt{B} \sqrt{ 1 + \frac{B \|\theta\|^2_\infty}{\|u\|^2_\infty} },$$
 we obtain
	\begin{align}
	1 &= \lim_{i\to\infty} \left( \frac{\|u\|_{a_{i+1}}}{\|u\|_\infty} \right)^{\frac{a_{i+1}}{\gamma^{i+1}}} \nonumber\\
	&\leq \exp\left\{ \frac{4n\sqrt{\gamma}}{(n-2)(\gamma-1)} \right\} \left(1 + \sqrt{4C_s}B^{\frac{1}{4}}\left(1 + \frac{B \|\theta\|^2_\infty}{\|u\|^2_\infty}\right)^{\frac{1}{4}}\right)^{\frac{4n}{(n-2)(\gamma-1)}} \left( \frac{\|u\|_{a_0}}{\|u\|_\infty} \right)^{a_0}.\label{E4}
\end{align}
Now observe that
$$\frac{4n}{(n-2)(\gamma-1)}=2\gamma_0=\frac{4pn}{2p-n}.$$
Also, from Lemma  \ref{K3},  the $L^2$-norm of $u=|\na\theta|$ satisfies
\begin{equation}\label{E5}
\|u\|^2_2=\|\nabla\theta\|^2_2\leq \lambda \|\theta\|_2^2.
\end{equation}
Furthermore, by interpolation,
	$$\|u\|_{a_0} \leq \|u\|^{1-\frac{1}{p}}_2 \|u\|^{\frac{1}{p}}_\infty.$$ 
Substituting this and (\ref{E5}) into (\ref{E4}), we finally deduce that
	\begin{equation}\label{finalinMoser}
\|u\|^2_\infty \leq \exp\left\{ \frac{4n\sqrt{\gamma}}{(n-2)(\gamma-1)} \right\} \left(1 + \sqrt{4C_s}B^{\frac{1}{4}}\left(1 + \frac{B \|\theta\|^2_\infty}{\|u\|^2_\infty}\right)^{\frac{1}{4}}\right)^{\frac{4pn}{2p-n}} \lambda\|\theta\|^2_2.
\end{equation} 
We now convert  \eqref{finalinMoser} into the desired estimate (\ref{E2}). Let
$$t = 4 C_s\sqrt{B}\sqrt{1+BD^{2}},\quad C(n,p)=\exp\left\{ \frac{4n\sqrt{\gamma}}{(n-2)(\gamma-1)} \right\}.$$
We distinguish two cases.\\	
\textbf{Case 1:} $D^{2}\|u\|^{2}_\infty \geq \|\theta\|^{2}_\infty$. Then $\frac{\|\theta\|_\infty^2}{\|u\|_\infty^2} \le D^2$. Substituting this into \eqref{finalinMoser} directly gives:
$$	\|u\|^2_\infty 
\leq C(n,p) (1 + \sqrt{t})^{\frac{4pn}{2p-n}} \lambda \|\theta\|^2_2.
$$
\noindent\textbf{Case 2:} $D^{2}\|u\|^{2}_\infty \leq \|\theta\|^{2}_\infty$. Then $\frac{\|u\|_\infty^2}{\|\theta\|_\infty^2} \le \frac{1}{D^2}$. Multiplying both sides of \eqref{finalinMoser} by $\|u\|_\infty^{\frac{2pn}{2p-n}}\|\theta\|_\infty^{-\frac{2pn}{2p-n}}$ and substituting the estimate for $\|\theta\|_\infty$ from Lemma \ref{L4} yields:
	\begin{align*}
	\|u\|^2_\infty \|u\|_\infty^{\frac{2pn}{2p-n}} 
	&\leq C(n,p)\left( \sqrt{\frac{\|u\|_\infty}{\|\theta\|_\infty}} + \sqrt{4C_s }B^{\frac{1}{4}} \left(\frac{\|u\|^2_\infty}{\|\theta\|^2_\infty} + B \right)^{\frac{1}{4}}\right)^{\frac{4pn}{2p-n}} \lambda \|\theta\|^2_2 \|\theta\|^{\frac{2pn}{2p-n}}_\infty \\
	&\leq C(n,p) \left(\frac{1}{\sqrt{D}}+ \sqrt{4 C_s} B^{\frac{1}{4}}\left(\frac{1}{D^2}+B\right)^{\frac{1}{4}}  \right)^{\frac{4pn}{2p-n}}\lambda \|\theta\|^2_2 \|\theta\|^{\frac{2pn}{2p-n}}_\infty \\
	&= C(n,p) (1 + \sqrt{t})^{\frac{4pn}{2p-n}} D^{-\frac{2pn}{2p-n}}\lambda \|\theta\|^2_2 \|\theta\|^{\frac{2pn}{2p-n}}_\infty\\
	&\leq C(n,p) (1 + \sqrt{t})^{\frac{4pn}{2p-n}} D^{-\frac{2pn}{2p-n}}\lambda \exp\left[C(n)\frac{2pn}{2p-n}\sqrt{\la}C_{s} \right] \|\theta\|^{2+\frac{2pn}{2p-n} }_2.
	\end{align*}
Combining the above two cases,  we arrive at the estimate \eqref{E2}. Here, the constant $C(n,p) > 0$ has been appropriately redefined to a slightly larger universal constant. This completes the proof.
\end{proof}

\subsection{Harnack Inequality for the Dual $1$-form}
 In this subsection, we establish a Harnack-type inequality for the  $1$-form $\theta$ dual to  a conformal Killing vector field. This inequality is the key analytic tool that allows us to compare the infimum and supremum of $|\theta|^{2}$, thereby proving that any nontrivial conformal Killing vector field is nowhere vanishing under the curvature assumptions of Theorem \ref{T2}.
	
The argument is a modification of the approach in \cite[Section 3.3]{HW}, adapted to the present setting of $1$-forms dual to conformal Killing vector fields. We first recall a classical mean value theorem for smooth functions on closed Riemannian manifolds.
\begin{theorem}\label{T6}(cf.\cite{Heb})
	Let $(M,g)$ be a connected, closed Riemannian manifold with diameter $D$, and let $f:M\rightarrow\mathbb{R}$ be a smooth function. Then for any $p,q\in M$,
	\begin{equation}\label{meanvalue}
	|f(p)-f(q)| \leq D\|\nabla f\|_{\infty}.
	\end{equation}
\end{theorem}
We now introduce the constant that will appear in the Harnack inequality. To simplify notation, define
\begin{equation}\label{E10}
I=C(n,p) (1 + \sqrt{t})^{\a} \sqrt{\lambda D^{2}},\quad II=C(n,p) (1 + \sqrt{t})^{\b} (\sqrt{\lambda D^{2}})^{\frac{\b}{\a} } e^{\left[C(n,p)\sqrt{\la}C_{s}\right]}.
\end{equation}
and set
$$\varepsilon(n,p,\kappa,\lambda,K,D)=\max\{I,II\}.$$ 
This constant  $\varepsilon(n,p,\kappa,\lambda,K,D)$ is precisely the quantity (up to the factor $D$) on the right-hand side of Proposition \ref{P6}, which asserts that
\begin{equation}\label{E6}
D\|\na\theta\|_{\infty}\leq \varepsilon(n,p,\kappa,\lambda,K,D)\|\theta\|_{2}.
\end{equation}
With these preparations, we can now prove the main result of this subsection.

\begin{theorem}\label{T4}
Let $(M,g)$ be a closed $n$-dimensional smooth Riemannian manifold satisfying the Sobolev inequality \eqref{Sobolev32}. Assume the same curvature conditions as in Proposition \ref{P6}. Suppose that $\theta$ is the $1$-form dual to a conformal Killing vector field, normalized by $\|\theta\|_2=1$. Then the following Harnack inequality holds:
	\begin{equation}\label{Harnack111}
	\frac{\inf|\theta|}{\sup|\theta|} \geq 1 -\varepsilon(n,p,\kappa,\lambda,K,D).
	\end{equation}
\end{theorem}
\begin{proof}
	Let $f = |\theta|$. Then $f$ is a smooth nonnegative function on $M$. Since $\nabla f = \nabla|\theta|$, and Kato’s inequality gives $|\nabla|\theta|| \leq |\nabla\theta|$, it follows that 
	\[
	|\nabla f| \leq |\nabla\theta|.
	\]
Applying the mean value inequality	(Theorem \ref{T6}) to the function $f$, we obtain 
\begin{equation}\label{meanvalueuse111}
	\sup f - \inf f \leq D\|\nabla\theta\|_{\infty} .
	\end{equation}
Now, since $\|\theta\|_{2}=1$, we have 
	\begin{equation}\label{meanvalueuse222}
\sup f\geq\|\theta\|_{2}=1
	\end{equation}
Combining \eqref{meanvalueuse111} and \eqref{meanvalueuse222}, and applying Proposition \ref{P6}, we finally deduce
	 \begin{align*}
	\inf f \geq &\sup f\cdot\left(1-D\|\nabla\theta\|_{\infty} \right)\\
	\geq&\sup f\cdot\left[1-\varepsilon(n,p,\kappa,\lambda,K,D)\right],
	\end{align*}
	which completes the proof.
\end{proof}

\section{ Proof of the Main Results}
In this final section, we combine the Harnack inequality established in Theorem \ref{T4} with an estimate for the Sobolev constant to prove the main rigidity theorem (Theorem \ref{T2}).

We begin by recalling a Sobolev–Poincaré inequality with an explicit dependence on the curvature and diameter bounds. This result is standard; we refer to \cite{Ber,Che} for proofs, and to \cite[Lemma 5.6]{Che} and \cite[Proposition 2.4]{HW} for detailed discussions of the Sobolev constant.
	\begin{proposition}\label{P1}
	Let $(M,g)$ be a closed $n$-dimensional smooth Riemannian manifold satisfying
		\[
		\operatorname{Ric} \geq -(n-1)\kappa \quad \text{and} \quad \operatorname{diam}(M) \leq D.
		\]
Then for every $f \in H^1(M)$, the following Sobolev inequality holds:
		\[
		\|f\|_{\frac{2n}{n-2}} \leq \|f\|_2 + C(n) D e^{(n-1)\sqrt{\kappa D^{2}}}  \|df\|_2.
		\]
 where $C(n)$ is a positive constant depends only on $n$.
	\end{proposition}
Proposition \ref{P1} provides an explicit form for the Sobolev constant $C_{s}$ in (\ref{Sobolev32}):
\[C_{s}=C(n) D e^{(n-1)\sqrt{\kappa D^{2}}}. \]
We are now in a position to prove our main theorem.
\begin{proof}[\textbf{Proof of Theorem \ref{T2}}]
Suppose that $X\not\equiv0$ is  a nontrivial conformal Killing vector field on $M$. Let $\theta=X^\flat$ be its dual $1$-form. Without loss of generality, we may normalize so that $\|\theta\|_2=1$.
	
We begin with the Harnack inequality (\ref{Harnack111}) from Theorem \ref{T4}. It remains to show that, for the choice of  $\la$ in Theorem \ref{T2}, both $I$ and $II$ (defined in (\ref{E10})) are bounded above by $\frac{1}{4}$.

Using the Sobolev inequality with explicit constant from Proposition \ref{P1}, we recall 
\begin{align}
t= 4 C_s\sqrt{B}\sqrt{1+BD^{2}}\leq 4C(n) e^{(n-1)\sqrt{\kappa D^{2}}} \sqrt{BD^2}\sqrt{1+BD^2}.\label{estimatefortRou}
\end{align}
Since $\sqrt{\lambda D^2}\leq e^{-(n-1)\sqrt{\kappa D^{2}}}$, there exists a positive constant $C_1(n,p)$ such that 
$$e^{\left[C(n,p)\sqrt{\la}C_{s}\right]}\leq  e^{C(n,p)C(n)}:=C_1(n,p).$$
Moreover, for  $B=5n(\kappa+\lambda+K)$,
\begin{align*}
e^{(n-1)\sqrt{\kappa D^{2}}}(1+BD^2)=&e^{(n-1)\sqrt{\kappa D^{2}}}\left[1+5n(\kappa+\lambda+K)D^2\right]\\
\leq&e^{(n-1)\sqrt{\kappa D^{2}}}\left[1+5n(\kappa D^2+e^{-2(n-1)\sqrt{\kappa D^{2}}}+KD^2)\right]\\
\leq&20ne^{2(n-1)\sqrt{\kappa D^{2}}}\left(1+KD^2\right).
\end{align*}
Hence, applying \eqref{estimatefortRou} and setting $C_{2}(n,p):=1+2\sqrt{20nC(n)}$, we have
\begin{align*}
	1+\sqrt{t}
	&\leq 1+\left[4C(n) e^{(n-1)\sqrt{\kappa D^{2}}} (1+BD^2)\right]^{\frac{1}{2}}\\
	&\leq 1+2\sqrt{20nC(n)}e^{(n-1)\sqrt{\kappa D^{2}}}(1+KD^2)^{\frac{1}{2}}\\
	&\leq C_2(n,p)e^{(n-1)\sqrt{\kappa D^{2}}}\left(1+\sqrt{K}D\right).
	\end{align*}	
Take
$$\tilde{C}(n,p) = \min \left\{ \frac{1}{[4C(n,p)]^{1/\alpha}C_2(n,p)}, \; \frac{1}{\left[ 4C(n,p) C_1(n,p) \right]^{1/\beta} C_2(n,p)} \right\}$$
and 
note that
$$\sqrt{\lambda D^2}= \min\left\{ 
		 \left(\frac{\tilde{C}(n,p)}{1 + \sqrt{K}D }e^{-(n-1)\sqrt{\kappa D^2}} \right)^{\frac{2pn}{2p-n}},  
\; e^{-(n-1)\sqrt{\kappa D^2}}
		\right\}.$$ 
Then we can estimate the terms $I$ and $II$ as follows:
\begin{itemize}
  \item For $I$, substituting the bounds for $(1+\sqrt{t})$ and $\sqrt{\lambda D^2}$:
  \begin{align*}
  I \leq& C(n,p) \cdot \left(C_2(n,p)e^{(n-1)\sqrt{\kappa D^{2}}}(1+\sqrt{K}D)\right)^\alpha \cdot  \left(\frac{\tilde{C}(n,p)}{1 + \sqrt{K}D }e^{-(n-1)\sqrt{\kappa }D} \right)^{\frac{2pn}{2p-n}}\\
  \leq& C(n,p) C_2(n,p)^\alpha \left( \frac{1}{4C(n,p)C_2(n,p)^\alpha} \right) = \frac{1}{4}.
  \end{align*}
  where the last inequality we use $\tilde{C}(n,p)\leq  \frac{1}{[4C(n,p)]^{1/\alpha}C_2(n,p)}$ and $\a=\frac{2pn}{2p-n}$.
  \item For $II$, substituting the bounds for $(1+\sqrt{t})$, $\sqrt{\lambda D^2}$ and $e^{\left[C(n,p)\sqrt{\la}C_{s}\right]}$:
\begin{align*}
II\leq &C(n,p) C_1(n,p) \left(C_2(n,p)e^{(n-1)\sqrt{\kappa D^{2}}}(1+\sqrt{K}D)\right)^\beta\cdot  \left(\frac{\tilde{C}(n,p)}{1 + \sqrt{K}D }e^{-(n-1)\sqrt{\kappa }D} \right)^{\frac{2pn}{2p-n}\cdot\frac{\beta}{\alpha}}\\
\leq&C(n,p) C_1(n,p)C_2(n,p)^\beta\cdot \left\{\frac{1}{\left[ 4C(n,p) C_1(n,p) \right]^{1/\beta} C_2(n,p)}\right\}^{\beta}=\frac{1}{4}.
\end{align*}
  where the last inequality we use $\tilde{C}(n,p)\leq \frac{1}{\left[ 4C(n,p) C_1(n,p) \right]^{1/\beta} C_2(n,p)}$ and $\a=\frac{2pn}{2p-n}$.
\end{itemize}
Combining the estimates for $I$ and $II$, we finally obtain $$\varepsilon(n,p,\kappa,\lambda,K,D)=\max\{I,II\}\leq \frac{1}{4},$$
which together with \eqref{Harnack111} yields
	\begin{equation*}
		\frac{\inf|\theta|}{\sup|\theta|} \geq 1 -\varepsilon(n,p,\kappa,\lambda,K,D)\geq \frac{3}{4}.
	\end{equation*} 
This Harnack inequality implies that if $X\not\equiv0$, then $X$ is nowhere vanishing.
\end{proof}
\begin{proof}[\textbf{Proof of Theorem \ref{T3}}]
This follows immediately from Theorem \ref{T2} together with the Poincaré–Hopf theorem. Indeed, if $M$ is a closed even-dimensional manifold with nonvanishing Euler characteristic $\chi(M)\neq0$, then the Poincaré–Hopf theorem implies that every vector field on $M$ must have at least one zero. By Theorem \ref{T2}, any conformal Killing vector field on $M$ is either identically zero or nowhere vanishing. Since a nowhere vanishing vector field cannot exist when $\chi(M)\neq 0$, we conclude that every conformal Killing vector field on $M$ must vanish identically.
\end{proof} 
\begin{proof}[\textbf{Proof of Corollary \ref{C1}}]
From Theorem \ref{T3}, we conclude that the Lie algebra of $\operatorname{Con}(M,g)$ is trivial, which implies that $\operatorname{Con}(M,g)$ is a discrete group.  
		
By the theorems of Lelong-Ferrand \cite{Lelong} and Obata \cite{Obata}, $\operatorname{Con}(M,g)$ is compact unless the manifold $M$ is conformally equivalent to the standard round sphere $S^{2n}$. However, $\operatorname{Con}(S^{2n},g_{round})$ admits a nontrivial Lie algebra, so this exceptional case cannot occur under the assumptions of Theorem \ref{T2}. Therefore, $\operatorname{Con}(M,g)$ is a compact discrete group, and hence finite.
\end{proof}
\begin{proof}[\textbf{Proof of Corollary \ref{C3}}]
Fix a point  $p\in M$, and consider the stabilizer subgroup
	$$\operatorname{Con}_p = \{x \in \operatorname{Con}(M,g) \mid x \cdot p = p\}$$
This is a closed Lie subgroup of $\operatorname{Con}(M,g)$. By the infinitesimal action, each element in the Lie algebra of $\operatorname{Con}_p$ corresponds to a conformal Killing vector field that vanishes at $p$. 
	
By Theorem \ref{T2}, any such vector field must vanish identically on $M$. Hence, the Lie algebra of $\operatorname{Con}_p$ is trivial, which implies that the stabilizer subgroup $\operatorname{Con}_p$ is discrete. Since $\operatorname{Con}(M,g)$ is compact (by Corollary \ref{C1}), every discrete subgroup is finite. Thus, the action of $\operatorname{Con}(M,g)$ on $M$ is locally free.
\end{proof}
\begin{proof}[\textbf{Proof of Corollary \ref{C2}}]
This follows from Bott's result \cite{Bot} (see also Corollary \ref{C4}), which states that if a closed oriented Riemannian manifold admits a nowhere vanishing Killing vector field, then all its Pontryagin numbers vanish. Under the topological assumptions of Corollary \ref{C2} (either $\chi(M)\neq 0$ in even dimensions, or at least one Pontryagin number is nonzero in dimensions divisible by 4), a nowhere vanishing Killing vector field cannot exist. By Theorem \ref{T2}, every Killing vector field must be identically zero. Hence, the isometry group $\operatorname{Iso}(M,g)$ has a trivial Lie algebra and is therefore finite. 
\end{proof}
\appendix

\section{Pontryagin Numbers and Bott's Theorem}\label{appendix B}
We first recall some basic facts about Pontryagin numbers. Let $M$ be a closed, oriented smooth manifold of dimension $4m$. The real tangent bundle $TM$ admits Pontryagin classes
\[
p_k(TM) := (-1)^k c_{2k}(TM \otimes_{\mathbb R} \mathbb C) \in H^{4k}(M; \mathbb Z), \qquad k = 1, \dots, m.
\]
For any sequence of non-negative integers $(i_1, \dots, i_r)$ such that 
$$i_1 + \cdots + i_r = m,$$
the corresponding \emph{Pontryagin number} is defined by
\[
p_{i_1 i_2 \cdots i_r}(M) := \big\langle p_{i_1}(TM) \cup p_{i_2}(TM) \cup \cdots \cup p_{i_r}(TM), [M] \big\rangle \in \mathbb Z,
\]
where $[M]$ denotes the fundamental class of $M$. These numbers are diffeomorphism invariants and, via the Chern--Weil theory, can be expressed as integrals of polynomials in the curvature forms of any Riemannian metric on $M$.	For example:
\begin{itemize}
	\item[(1)] By the Hirzebruch signature theorem, the signature $\sigma(M)$ of a $4m$-dimensional manifold is a linear combination of its Pontryagin numbers:
	\[
	\sigma(M) = \langle L(TM), [M] \rangle,
	\]
	where $L(TM)$ is the Hirzebruch $L$-genus, a polynomial in the Pontryagin classes. 
	\item[(2)] The $\hat{A}$-genus $\hat{A}(M)$ is also a polynomial in the Pontryagin classes. By the Lichnerowicz theorem, if a closed spin manifold admits a metric of positive scalar curvature, then its $\hat{A}$-genus vanishes.
\end{itemize}
We now recall a classical result of Bott \cite{Bot} that relates Pontryagin numbers to the existence of nowhere vanishing Killing vector fields. In the Riemannian setting, Bott proved the following theorem:
\begin{theorem}(\cite[Theorem 2]{Bot})
	Let $(M,g)$ be a closed, oriented even-dimensional Riemannian manifold, and let $X$ be a non-degenerate Killing vector field on $M$. Then for every polynomial $\Phi$ in the Chern classes of the complexified tangent bundle of weight at most $\dim M / 2$, one has
	\[
	\sum_{P \in \operatorname{Zero}(X)} \frac{\Phi(\mathbf L_P)}{\det^{1/2}(\mathbf L_P)} = \Phi(M),
	\]
	where $P$ ranges over the zeros of $X$, $\mathbf L_P$ denotes the linearization of $X$ at $P$, $\det^{1/2}(\mathbf L_P)$ is the square root determined by the orientation, and $\Phi(M)$ is the characteristic number of $M$ defined by $\Phi$.
\end{theorem}
In particular, if $X$ is \emph{nowhere vanishing}, then the left-hand side is an empty sum and hence $\Phi(M) = 0$ for all such $\Phi$. Therefore, we obtain the following fundamental corollary:
\begin{corollary}[\cite{Bot}]\label{C4}
If a closed Riemannian manifold $(M,g)$ admits a nowhere vanishing Killing vector field, then all Pontryagin numbers of $M$ vanish.
\end{corollary} 

\section{Proof of the Integral Inequality for $\na\theta$}\label{appendixA}
In this appendix, we provide the detailed proof of Lemma \ref{L7}, which is the key integral inequality used in the Moser iteration for $|\na\theta|$. 

We first establish two divergence identities that will be used repeatedly in the proof.  Set $u:=|\na\theta|$. Since the computations are pointwise, It suffices to work at a fixed point a point $p\in M$ and choose a local orthonormal frame  $\{e_i\}_{i=1}^n$  in a neighborhood of $p$ such that
$$\left.\nabla_{e_i}e_j\right|_p=0\quad \forall 1\leq i,j\leq n.$$

\begin{lemma}\label{divergence1}
Suppose that $Y_1\in\mathfrak{X}(M)$ is a vector field on $M$ defined by
$$Y_1:=\sum_{i=1}^nu^{2k-2}\langle R(e_i,\cdot)\theta,\nabla_{e_i}\theta\rangle^{ \flat}.$$
Then the divergence of  $Y_{1}$ is 
\begin{equation}\label{diver1}
\mathrm{div}\,Y_1=(2k-2)u^{2k-3}\langle R(\cdot,\nabla u)\theta,\nabla \theta\rangle+u^{2k-2}\langle \nabla^*[R(\cdot,\cdot)\theta],\nabla \theta\rangle-\frac{1}{2}u^{2k-2}|R\theta|^2.
\end{equation}
\end{lemma}
\begin{proof}
Denote $\tilde{Y}_1=\sum_{i=1}^n\limits\langle R(e_i,\cdot)\theta,\nabla_{e_i}\theta\rangle^{\flat}$, so that $Y_1= u^{2k-2}\tilde{Y}_1$.  For any vector field $Z\in\mathfrak{X}(M)$, we have
$$\langle \tilde{Y}_1,Z\rangle=\langle R(e_i,Z)\theta,\nabla_{e_i}\theta\rangle.$$
Then, using the product rule for divergence,
$$
\begin{aligned}
\mathrm{div}\,Y_1=&\sum_{j=1}^n\left\langle \nabla_{e_j} \left(u^{2k-2}\tilde{Y}_1\right),e_j\right\rangle\\
=&\sum_{j=1}^n(2k-2)u^{2k-3}\left\langle \tilde{Y}_{1}\na_{e_{j}}u,e_j\right\rangle+\sum_{j=1}^nu^{2k-2}\left\langle \nabla_{e_j} \tilde{Y}_1,e_j\right\rangle \\
=&\sum_{j=1}^n(2k-2)u^{2k-3}\left\langle \tilde{Y}_1,(\na_{e_{j}}u)e_j\right\rangle +\sum_{j=1}^nu^{2k-2}\left\langle \nabla_{e_j} \tilde{Y}_1,e_j\right\rangle\\
=&(2k-2)u^{2k-3}\left\langle \tilde{Y}_1,\nabla u\right\rangle +\sum_{j=1}^nu^{2k-2}\nabla_{e_j}\left\langle  \tilde{Y}_1,e_j\right\rangle
\end{aligned}
$$
Then the lemma follow from the relation
$$
\begin{aligned}
(2k-2)u^{2k-3}\left\langle \tilde{Y}_1,\nabla u\right\rangle=&(2k-2)u^{2k-3}\sum_{i=1}^n\langle R(e_i,\nabla u)\theta,\nabla_{e_i}\theta\rangle\\
=&(2k-2)u^{2k-3}\langle R(\cdot,\nabla u)\theta,\nabla \theta\rangle
\end{aligned}
$$
and
\begin{align*}
\sum_{j=1}^nu^{2k-2}\nabla_{e_j}\left\langle  \tilde{Y}_1,e_j\right\rangle=&\sum_{i,j=1}^nu^{2k-2}\nabla_{e_j}\langle R(e_i,e_j)\theta,\nabla_{e_i}\theta\rangle\\
=&\sum_{i=1}^nu^{2k-2}\langle \nabla_{e_j}[R(e_i,e_j)\theta],\nabla_{e_i}\theta\rangle+\sum_{i,j=1}^nu^{2k-2}\langle R(e_i,e_j)\theta,\nabla_{e_j}\nabla_{e_i}\theta\rangle\\
=& -\sum_{i=1}^nu^{2k-2}\langle \nabla_{e_j}[R(e_j,e_i)\theta],\nabla_{e_i}\theta\rangle \\
&+\frac{1}{2}\left(\sum_{i,j=1}^nu^{2k-2}\langle R(e_i,e_j)\theta,\nabla_{e_j}\nabla_{e_i}\theta\rangle+ \sum_{i,j=1}^nu^{2k-2}\langle R(e_j,e_i)\theta,\nabla_{e_i}\nabla_{e_j}\theta\rangle \right) \\
=&u^{2k-2}\langle \nabla^*[R(\cdot,\cdot)\theta],\nabla \theta\rangle+\frac{1}{2}\sum_{i,j=1}^{n}u^{2k-2}\langle R(e_i,e_j)\theta,(\nabla_{e_j}\nabla_{e_i}-\nabla_{e_i}\nabla_{e_j})\theta\rangle  \\
=&u^{2k-2}\langle \nabla^*[R(\cdot,\cdot)\theta],\nabla \theta\rangle+\frac{1}{2}\sum_{i,j=1}^{n}u^{2k-2}\langle R(e_i,e_j)\theta,R(e_j,e_i)\theta\rangle  \\
=&u^{2k-2}\langle \nabla^*[R(\cdot,\cdot)\theta],\nabla \theta\rangle-u^{2k-2}\frac{1}{2}|R\theta|^2.
\end{align*}
Here we have used the curvature definition $R(e_i,e_j)\theta=\na_{e_i}\na_{e_j}\theta-\na_{e_j}\na_{e_i}\theta$, and the identity  $|R\theta|^2=\sum_{i,j=1}^n\limits\langle R(e_i,e_j)\theta,R(e_i,e_j)\theta\rangle$. 
\end{proof}
\begin{lemma}\label{divergence2}
Suppose that $Y_2\in\mathfrak{X}(M)$ is a vector field on $M$ defined by
$$Y_2:=u^{2k-2}\left\langle \mathrm{Ric}(\theta),\nabla_{(\cdot)}\theta\right\rangle^{\flat}$$
Then the divergence of $Y_2$ is
\begin{equation}\label{diver2}
\mathrm{div}\,Y_2=u^{2k-2}\langle \nabla\mathrm{Ric}(\theta),\nabla \theta\rangle-u^{2k-2}\langle \mathrm{Ric}(\theta),\nabla^*\nabla \theta\rangle+(2k-2)u^{2k-3}\left\langle \mathrm{Ric}(\theta),\nabla_{\nabla u} \theta\right\rangle.
\end{equation} 
\end{lemma}
\begin{proof}
Denote $\tilde{Y}_2=\langle \mathrm{Ric}(\theta),\nabla_{(\cdot)}\theta\rangle^{\flat}$, so that $Y_2= u^{2k-2}\tilde{Y}_2$. For any vector field $Z\in \mathfrak{X}(M)$, we have
$$\langle \tilde{Y}_2,Z\rangle=\langle  \mathrm{Ric}(\theta),\nabla_{Z}\theta\rangle.$$
Then, similar to Lemma \ref{diver1},
$$
\begin{aligned}
\mathrm{div}\,Y_2
=&(2k-2)u^{2k-3}\left\langle \tilde{Y}_2,\nabla u\right\rangle +\sum_{j=1}^nu^{2k-2}\nabla_{e_j}\left\langle  \tilde{Y}_2,e_j\right\rangle
\end{aligned}
$$
Then the lemma follow from the relation
$$
\begin{aligned}
(2k-2)u^{2k-3}\left\langle \tilde{Y}_2,\nabla u\right\rangle
=(2k-2)u^{2k-3}\left\langle \mathrm{Ric}(\theta),\nabla_{\nabla u} \theta\right\rangle
\end{aligned}
$$
and
\begin{align*}
\sum_{j=1}^nu^{2k-2}\nabla_{e_j}\left\langle  \tilde{Y}_2,e_j\right\rangle=&\sum_{j=1}^nu^{2k-2}\nabla_{e_j}\left\langle \mathrm{Ric}(\theta),\nabla_{e_j}\theta\right\rangle\\
=&\sum_{j=1}^nu^{2k-2}\langle \nabla_{e_j}[\mathrm{Ric}(\theta)],\nabla_{e_j}\theta\rangle+\sum_{j=1}^nu^{2k-2}\langle \mathrm{Ric}(\theta),\nabla_{e_j}\nabla_{e_j}\theta\rangle\\
=&u^{2k-2}\langle \nabla\mathrm{Ric}(\theta),\nabla \theta\rangle-u^{2k-2}\langle \mathrm{Ric}(\theta),\nabla^*\nabla \theta\rangle.
\end{align*}
\end{proof}
Integrating the divergence identities \eqref{diver1} over $M$ yields the following integral identities.
\begin{proposition}\label{diverintegral1}
Let $(M,g)$ be a closed $n$-dimensional smooth Riemannian manifold and $u=|\nabla \theta|$. Then for $k>1$, the following integral inequality holds:
$$\int_M \langle \nabla^* \left[R(\cdot,\cdot)\theta\right], u^{2(k-1)}\nabla\theta \rangle =-(2k-2) \int_M u^{2k-3}\langle R(\cdot,\nabla u)\theta,\nabla \theta\rangle+\frac{1}{2}\int_M u^{2k-2}|R\theta|^2.$$
\end{proposition}
%
%
From the  the divergence identities \eqref{diver1}, we then have
\begin{proposition}\label{diverintegral2}
Let $(M,g)$ be a closed $n$-dimensional smooth Riemannian manifold and $u=|\nabla \theta|$. Then for $k>1$, the following integral inequality holds:
\begin{align*}
\int_M \left\langle\nabla\mathrm{Ric}(\theta),\nabla \theta\right\rangle u^{2k-2}
=&\int_M \left|\mathrm{Ric}(\theta)\right|^2 u^{2k-2}+(2k-2)\cdot\frac{n-2}{n}\int_M \left\langle\mathrm{Ric}(\theta),du\right\rangle\delta \theta\cdot u^{2k-3}\\
&-\frac{n-2}{n}\int_M \delta \theta\cdot\delta\left[\mathrm{Ric}(\theta)\right] u^{2k-2}-(2k-2)\int_M \left\langle\mathrm{Ric}(\theta),\nabla_{\nabla u} \theta\right\rangle u^{2k-3}.
\end{align*}
\end{proposition}
\begin{proof}
Integrating both sides of \eqref{diver2} in Lemma \ref{divergence2} and applying divergence theorem yields
$$\int_M \langle \nabla[\mathrm{Ric}(\theta)], u^{2k-2}\nabla \theta \rangle =\int_M u^{2k-2}\langle \mathrm{Ric}(\theta),\nabla^*\nabla \theta\rangle-(2k-2) \int_M u^{2k-3}\langle \mathrm{Ric}(\theta),\nabla_{\nabla u} \theta\rangle.$$
Substituting \eqref{mainKillForm1}, i.e., $\na^{\ast}\na\theta=\mathrm{Ric}(\theta)-\frac{n-2}{n}d\delta\theta$, the integral becomes
\begin{align*}
\int_M \left\langle\nabla\mathrm{Ric}(\theta),\nabla \theta\right\rangle u^{2k-2}
=&\int_M \left\langle\mathrm{Ric}(\theta),\mathrm{Ric}(\theta)\right\rangle u^{2k-2}-\int_M\frac{n-2}{n}\left\langle \mathrm{Ric}(\theta), d\delta \theta\right\rangle u^{2k-2}\\
&-(2k-2)\int_M \left\langle\mathrm{Ric}(\theta),\nabla_{\nabla u} \theta\right\rangle u^{2k-3}\\
=&\int_M \left|\mathrm{Ric}(\theta)\right|^2 u^{2k-2}+(2k-2)\cdot\frac{n-2}{n}\int_M \left\langle\mathrm{Ric}(\theta),du\right\rangle\delta \theta\cdot u^{2k-3}\\
&-\frac{n-2}{n}\int_M \delta \theta\cdot\delta\left[\mathrm{Ric}(\theta)\right] u^{2k-2}-(2k-2)\int_M \left\langle\mathrm{Ric}(\theta),\nabla_{\nabla u} \theta\right\rangle u^{2k-3}
\end{align*}
where we use the relation $\delta\left[u^{2k-2}\mathrm{Ric}(\theta)\right]= u^{2k-2}\delta\left[\mathrm{Ric}(\theta)\right]-(2k-2)u^{2k-3}\left\langle\mathrm{Ric}(\theta),du\right\rangle$.
\end{proof}
Based on these integral identities in Propositions \ref{diverintegral1} and \ref{diverintegral2}, we now prove the Lemma \ref{L7}.
\begin{proof}[\textbf{The proof of Lemma \ref{L7}}]
Multiplying both sides of \eqref{InequalityforGra} by the function $u^{2k-2}$ and then integrating yields 
\begin{align}
 &\frac{1}{2}\int_M\Delta_d(u^2)u^{2k-2}+\int_M\left|\nabla^2 \theta\right|^2u^{2k-2}+\frac{n-2}{n(n-1)}\int_M\delta \theta\cdot\delta\left[\mathrm{Ric}(\theta)\right]u^{2k-2}\nonumber \\
=&\int_M \left\langle\nabla\mathrm{Ric}(\theta),\nabla \theta\right\rangle u^{2k-2}+\int_M\left\langle \nabla^*\left[R(\cdot,\cdot)\theta\right],\nabla \theta\right\rangle u^{2k-2}-\int_M\left\langle\mathrm{Ric}(\nabla \theta),\nabla \theta\right\rangle u^{2k-2}.\label{apendx111}
\end{align} 
 Noting that
$$	\sum_{i,j=1}^n \left\langle \nabla_{\mathrm{Ric}(e_j)}\theta, \nabla_{e_j} \theta \right\rangle=\sum_{i,j=1}^{n}\mathrm{Ric}(e_i,e_j)\langle\nabla_{e_i}\theta,\nabla_{e_j}\theta\rangle\geq {\mathrm{Ric}}^- |\nabla\theta|^2,
$$
by \eqref{Rictheta}, we first obtain 
\begin{equation}\label{apendx222}
-\int_M\left\langle\mathrm{Ric}(\nabla \theta),\nabla \theta\right\rangle u^{2k-2}\leq\int_M \underline{\mathrm{Ric}}^{-}u^{2k}+\int_M |R| u^{2k}.
\end{equation}
Denote
\begin{align}
F(\theta)=&\int_M\left\langle \nabla^*\left[R(\cdot,\cdot)\theta\right],\nabla \theta\right\rangle u^{2k-2}+\int_M \underline{\mathrm{Ric}}^{-} u^{2k}+\int_M |R| u^{2k}+\int_M \left|\mathrm{Ric}(\theta)\right|^2 u^{2k-2}\nonumber\\
&+(2k-2)\cdot\frac{n-2}{n}\int_M \left\langle\mathrm{Ric}(\theta),du\right\rangle\delta \theta\cdot u^{2k-3}-(2k-2)\int_M \left\langle\mathrm{Ric}(\theta),\nabla_{\nabla u} \theta\right\rangle u^{2k-3}.\label{E7}
\end{align}
Simplifying \eqref{apendx111} by Proposition \ref{diverintegral2} and inequality \eqref{apendx222} yields
\begin{equation}\label{keyproofLem31}
\frac{1}{2}\int_M\Delta_d(u^2)u^{2k-2}+\int_M\left|\nabla^2 \theta\right|^2u^{2k-2}\leq-\frac{n-2}{n-1}\int_M\delta \theta\cdot\delta\left[\mathrm{Ric}(\theta)\right]u^{2k-2}+F(\theta).
\end{equation}
Note that
\begin{equation*}
	\frac{1}{2}\int_M\Delta_d(u^2)u^{2k-2}=(2k-2)\int_M|\nabla u|^2u^{2k-2}.
	\end{equation*}
Hence, the inequlity \eqref{keyproofLem31} can be rewritten as
$$
 (2k-2)\int_M|\nabla u|^2u^{2k-2}+\int_M\left|\nabla^2 \theta\right|^2u^{2k-2}\leq-\frac{n-2}{n-1}\int_M\delta \theta\cdot\delta\left[\mathrm{Ric}(\theta)\right]u^{2k-2}+F(\theta).
$$
With this, using the relation (cf. \cite[Lemma 2.3]{Semm})
$$|d\delta \theta| \le \sqrt{n} |\nabla^2 \theta|,\quad |\delta\theta|\leq \sqrt{n}|\nabla \theta|$$
and 
\begin{equation}\label{E8}
|ab|\leq 4n|a|^2 + \frac{|b|^2}{16n}
\end{equation}
for real constants $a,b$,  then we have for any $k>1$,
\begin{align*}
-\int_M \delta \theta\cdot\delta\left[\mathrm{Ric}(\theta)\right] \cdot u^{2(k-1)} =&-\int_M\left\langle \mathrm{Ric}(\theta), d\left[\delta\theta\cdot u^{2(k-1)}\right]\right\rangle \\
=&-\int_M\left\langle \mathrm{Ric}(\theta),   d\delta\theta\right\rangle u^{2(k-1)}-(2k-2)\int_M\left\langle \mathrm{Ric}(\theta),  du\right\rangle\delta\theta\cdot u^{2k-3}\\
\leq& \int_{M}|\mathrm{Ric}(\theta)| \sqrt{n}|\na^{2}\theta|u^{2(k-1)}+(2k-2)\int_{M}  |\mathrm{Ric}(\theta)|  \sqrt{n}|\delta\theta|u^{2(k-1)} \\
\leq &4n\int_M |\mathrm{Ric}(\theta)|^2 u^{2(k-1)}+\frac{1}{16}\int_M\left|\nabla^2\theta\right|^2 u^{2(k-1)}\\
&+8n(k-1) \int_M |\mathrm{Ric}(\theta)|^2 \cdot u^{2(k-1)} +\frac{ (k-1)}{8} \int_M |\nabla u|^2 u^{2(k-1)}\\
\leq &4n(2k-1)\int_M |\mathrm{Ric}(\theta)|^2 u^{2(k-1)} \\
&+\frac{1}{16}\left( (2k-2)\int_M|\nabla u|^2u^{2k-2}+\int_M\left|\nabla^2 \theta\right|^2u^{2k-2}\right) \\
\leq&-\frac{n-2}{16(n-1)}\int_M\delta \theta\cdot\delta\left[\mathrm{Ric}(\theta)\right]u^{2k-2}+\frac{1}{16}F(\theta)\\
&+4n(2k-1) \int_M |\mathrm{Ric}(\theta)|^2 \cdot u^{2(k-1)}.
\end{align*}
which gives
\begin{equation}\label{deltaRic}
-\frac{n-2}{n-1}\int_M \delta \theta\cdot\delta\left[\mathrm{Ric}(\theta)\right] \cdot u^{2(k-1)}\leq \frac{16(n-2)}{15n-14}\left(4n(2k-1)\int_M |\mathrm{Ric}(\theta)|^2 u^{2(k-1)}+\frac{1}{16}F(\theta)\right).
\end{equation}
Since $\frac{16(n-2)}{15n-14}\leq \frac{16}{15}$ for any $n\geq2$, substituting \eqref{deltaRic} into \eqref{keyproofLem31} and taking the absolute value, it follows that
\begin{equation*}
\frac{1}{2}\int_M\Delta_d(u^2)u^{2k-2}+\int_M\left|\nabla^2 \theta\right|^2u^{2k-2}\leq\frac{16}{15}\cdot4n(2k-1)\int_M |\mathrm{Ric}(\theta)|^2 u^{2(k-1)}+\frac{16}{15}|F(\theta)|.
\end{equation*}
The following lemma provides an estimate for $F(\theta)$; its proof is included below.
\begin{lemma}\label{lemforFtheta}
Under the same assumptions as above, for $k>1$, we have
\begin{align}
|F(\theta)|\leq&8k \int_M |R\theta|^2 \cdot u^{2(k-1)} +\int_M \underline{\mathrm{Ric}}^{-} u^{2k}+\int_M |R| u^{2k}\nonumber \\
&+10nk\int_M |\mathrm{Ric}(\theta)|^2 \cdot u^{2(k-1)}+\frac{(k-1)}{2} \int_M |\nabla u|^2 u^{2(k-1)}.\label{E9}
\end{align}
\end{lemma}
\begin{proof}
From the definition of $F(\theta)$ in (\ref{E7}), we estimate each term separately.\\
\textbf{Term 1}: $\int_M\left\langle \nabla^*\left[R(\cdot,\cdot)\theta\right],\nabla \theta\right\rangle u^{2k-2}$. Using Proposition \ref{diverintegral1} and the relation
\begin{equation*}
|ab|\leq 4|a|^2 + \frac{|b|^2}{16}
\end{equation*}
for real constants $a,b$,  we observe that  for any $k>1$,
	\begin{align}
	\left|\int_M \langle \nabla^* \left[R(\cdot,\cdot)\theta\right], \nabla \theta \rangle u^{2(k-1)}\right|
	=& \left|-(2k-2) \int_M u^{2k-3}\langle R(\cdot,\nabla u)\theta,\nabla \theta\rangle+\frac{1}{2}\int_M u^{2k-2}|R\theta|^2\right| \nonumber\\
	\leq &(2k-2) \int_M |R\theta| \cdot |\nabla u| \cdot u^{2(k-1)}+ \frac{1}{2}\int_M |R\theta|^2 \cdot u^{2(k-1)}\nonumber \\
	\leq &\int_M |R\theta|^2 \cdot u^{2(k-1)}\nonumber\\
	& + (2k-2) \left( 4\int_M |R\theta|^2 \cdot u^{2(k-1)} + \frac{1}{16} \int_M |\nabla u|^2 u^{2(k-1)} \right)\nonumber \\
	\leq& 8k \int_M |R\theta|^2 \cdot u^{2(k-1)} +\frac{ (k-1)}{8} \int_M |\nabla u|^2 u^{2(k-1)}.\label{append6111}
	\end{align}
\textbf{Term 2}: $-(2k-2)\int_M \left\langle\mathrm{Ric}(\theta),\nabla_{\nabla u} \theta\right\rangle u^{2k-3}$. Similarly, using the relation
\begin{equation*}
|ab|\leq 2|a|^2 + \frac{|b|^2}{8}
\end{equation*}
for real constants $a,b$, we deduce that  for any $k>1$,
	\begin{align}
	\left|-(2k-2)\int_M \left\langle\mathrm{Ric}(\theta),\nabla_{\nabla u} \theta\right\rangle u^{2k-3}\right|\leq& (2k-2)\int_M|\mathrm{Ric}(\theta)| |\nabla u|u^{2(k-1)}\nonumber\\
	\leq &4k \int_M |\mathrm{Ric}(\theta)|^2 \cdot u^{2(k-1)} +\frac{ (k-1)}{4} \int_M |\nabla u|^2 u^{2(k-1)}\nonumber\\
	\leq &2nk \int_M |\mathrm{Ric}(\theta)|^2 \cdot u^{2(k-1)} +\frac{ (k-1)}{4} \int_M |\nabla u|^2 u^{2(k-1)}.
	\label{append6333}
	\end{align}
\textbf{Term 3}: $-(2k-2)\cdot\frac{n-2}{n}\int_M \left\langle\mathrm{Ric}(\theta),du\right\rangle\delta \theta\cdot u^{2k-3}$. Using $|\delta\theta|\leq \sqrt{n}|\nabla \theta|$ and the Peter-Paul inequality (\ref{E8}), we have for any $k>1$,
	\begin{align}
	&\left|-(2k-2)\cdot\frac{n-2}{n}\int_M \left\langle\mathrm{Ric}(\theta),du\right\rangle\delta \theta\cdot u^{2k-3}\right|\nonumber\\
	\leq& (2k-2)\int_M \left|\left\langle\mathrm{Ric}(\theta),du\right\rangle\delta \theta\cdot u^{2k-3}\right|\nonumber\\
	\leq &8n(k-1) \int_M |\mathrm{Ric}(\theta)|^2 \cdot u^{2(k-1)} +\frac{ (k-1)}{8} \int_M |\nabla u|^2 u^{2(k-1)}.\label{append6222}
	\end{align}
Combining \eqref{append6111}, \eqref{append6222} and \eqref{append6333}, and then taking absolute values, yields (\ref{E8}). 
\end{proof}
Finally, substituting the estimate for $F(\theta)$ from (\ref{E9}) into (\ref{deltaRic}), we obtain exactly the integral inequality stated in Lemma \ref{L7}. This completes the proof of Lemma \ref{L7}. 
\end{proof}

\section{Technical Lemmas for Moser Iteration}
We first recall the following Lemma, which is used in the proof of Lemma \ref{L4}.
	\begin{theorem}\label{B1}
	Let $r>0,\nu>1$. Let $\{A(i)\}_{i\geq0}$ be a nonnegative and increasing sequence. If 
	$$A(i+1)\leq (1 + r\sqrt{\nu^{i}})^{\frac{1}{\nu^{i}}} A(i),\quad \forall i\geq0,$$
	then
	$$\lim_{i\to +\infty}A(i)\leq\exp\left(\frac{r\sqrt{\nu}}{\sqrt{\nu}-1}\right)A(0).$$
	\end{theorem}
	\begin{proof}
		For any non-negaitive integer $i$, iterating we find that
	\begin{align*}
	A(i+1)\leq&(1 + r\sqrt{\nu^{i}})^{\frac{1}{\nu^{i}}} A(i)\leq (1 + r\sqrt{\nu^{i}})^{\frac{1}{\nu^{i}}}(1 + r\sqrt{\nu^{i-1}})^{\frac{1}{\nu^{i-1}}} A(i-1)\\
	\leq&\left(\prod_{l=0}^{i}(1 + r\sqrt{\nu^{l}})^{\frac{1}{\nu^{l}}}\right)A(0)=\left(\prod_{l=0}^{i}\exp\left(\frac{1}{\nu^{l}}\ln(1 + r\sqrt{\nu^{l}}\right)\right)A(0)\\
	\leq&\exp\left(\sum_{l=0}^i\frac{r}{\sqrt{\nu^l}}\right)A(0).
	\end{align*}
	Here we used the inequality $\ln(1+x)\leq x$ for $x>0$. Taking the limit as $i \to +\infty$, we have
	$$\lim_{i\to +\infty}A(i)\leq \exp\left(\sum_{l=0}^{+\infty}\frac{r}{\sqrt{\nu^l}}\right)A(0)=\exp\left(\frac{r\sqrt{\nu}}{\sqrt{\nu}-1}\right)A(0).$$
	This completes the proof.
	\end{proof}
	We next estimate the following infinite product.
	\begin{lemma}\label{L6}
		Let $r> 0$, $\gamma > 1$. Then
		\[
		P := \prod_{i=0}^\infty (1 + r\gamma^{i+1})^{\frac{1}{\gamma^{i+1}}}
		\leq \exp\left\{ \frac{2\sqrt{\gamma}}{\gamma-1} \right\} (1 + \sqrt{r})^{\frac{2}{\gamma-1}}.
		\]
	\end{lemma}
	
	\begin{proof}
		For any $x > 0$, we have
		\begin{align*}
		\ln(1 + rx) &\leq 2\ln(1 + \sqrt{rx}) \\
		&= 2\ln(1 + \sqrt{r}) + 2\ln\left( 1 + (\sqrt{x} - 1)\frac{\sqrt{r}}{1 + \sqrt{r}} \right) \\
		&\leq 2\ln(1 + \sqrt{r}) + 2(\sqrt{x} - 1).
		\end{align*}
		Hence,
		\begin{align*}
		\ln P &= \sum_{i=0}^\infty \frac{\ln(1 + r\gamma^{i+1})}{\gamma^{i+1}} \\
		&\leq \sum_{i=0}^\infty \frac{2\ln(1 + \sqrt{r})}{\gamma^{i+1}} + \sum_{i=0}^\infty \frac{2\gamma^{\frac{i+1}{2}} - 2}{\gamma^{i+1}} \\
		&\leq \frac{2\ln(1 + \sqrt{r})}{\gamma-1} + \frac{2}{\sqrt{\gamma} - 1} - \frac{2}{\gamma-1} \\
		&= \frac{2\ln(1 + \sqrt{r})}{\gamma-1} + \frac{2\sqrt{\gamma}}{\gamma-1}.
		\end{align*}
		Therefore, we complete the proof of  this lemma.
	\end{proof}
	Based on Lemma \ref{L6}, we prove the following technical lemma, which is used in Moser iteration procedure in Proposition \ref{P6}.
	\begin{theorem}\label{B2}
	Let $r,s>0,\gamma>1$. Let $\{A(i)\}_{i\geq0}$ be a nonnegative and increasing sequence. If 
	$$A(i+1)\leq (1 + r\gamma^{i+1})^{\frac{s}{\gamma^{i+1}}} A(i),\quad \forall i\geq0,$$
	then
	$$\lim_{i\to +\infty}A(i)\leq \exp\left\{ \frac{2s\sqrt{\gamma}}{\gamma-1} \right\} (1 + \sqrt{r})^{\frac{2s}{\gamma-1}}A(0).$$
	\end{theorem}
	\begin{proof}
	For any non-negaitive integer $i$, iterating we find that
	\begin{align*}
	A(i+1)\leq& (1 + r\gamma^{i+1})^{\frac{s}{\gamma^{i+1}}} A(i)\leq (1 + r\gamma^{i+1})^{\frac{s}{\gamma^{i+1}}} (1 +r\gamma^{i})^{\frac{s}{\gamma^{i}}} A(i-1)\\
	\leq&\left(\prod_{l=1}^{i+1} (1 + r\gamma^{l})^{\frac{1}{\gamma^{l}}}\right)^sA(0).
	\end{align*}
	Taking the limit as $i \to +\infty$ and applying the index shift $l = k+1$, we can apply Lemma \ref{L6} to get
	\begin{align*}
	\lim_{i\to +\infty}A(i)\leq& \left(\prod_{l=1}^{+\infty} (1 + r\gamma^{l})^{\frac{1}{\gamma^{l}}}\right)^sA(0)\\
	\leq &\exp\left\{ \frac{2s\sqrt{\gamma}}{\gamma-1} \right\} (1 + \sqrt{r})^{\frac{2s}{\gamma-1}}A(0).
	\end{align*}
	This completes the proof.
	\end{proof}
 
\subsection*{Acknowledgements}
This work is supported by the National Natural Science Foundation of China Nos. 12271496 and the Youth Innovation Promotion Association CAS.
	
\subsection*{Data Availability}
This manuscript has no associated data.
\subsection*{Declarations}
\textbf{Conflict of interest} The author states that there is no conflict of interest.

\end{document}